\documentclass[12pt,a4paper]{amsart}
\usepackage{amsmath,amsthm,amsfonts,amssymb,graphicx,psfrag,dsfont,a4wide}
\usepackage{color}

\usepackage{hyperref}
\usepackage[alphabetic,initials]{amsrefs}

\usepackage{soul}
\usepackage{enumerate}
\psfrag{S-}{$\Sigma_-$}
\psfrag{S+}{$\Sigma_+$}
\psfrag{Gam}{$\Gamma$}
\psfrag{S1 x}{$S^1 \times$}
\psfrag{(W,dlambda)}{$(W,d\lambda)$}
\psfrag{(V4,xi3)}{$(V_4,\xi_3)$}
\psfrag{(V2,xi2)}{$(V_2,\xi_2)$}

\theoremstyle{plain}

\newtheorem{thm}{Theorem}
\newtheorem{prop}{Proposition}[section]

\newtheorem{cor}{Corollary}
\newtheorem{lemma}[prop]{Lemma}

\newtheorem{question}{Question}
\newtheorem*{conju}{Conjecture}

\theoremstyle{definition}
\newtheorem{example}[prop]{Example}

\newtheorem{defn}[prop]{Definition}

\theoremstyle{remark}
\newtheorem{remark}[prop]{Remark}

\theoremstyle{plain}
\newtheorem{theorem}[prop]{Theorem}
\newtheorem{proposition}[prop]{Proposition}
\newtheorem{corollary}[prop]{Corollary}
\newtheorem{conjecture}[prop]{Conjecture}
\theoremstyle{definition}
\newtheorem{definition}[prop]{Definition}

\newcommand{\Crit}{\operatorname{Crit}}

\newcommand{\handle}{\mathcal{H}}

\newcommand{\ind}{\operatorname{ind}}

\newcommand{\interior}{\operatorname{int}}

\newcommand{\Lie}{\mathcal{L}}

\newcommand{\muCZ}{\mu_{\operatorname{CZ}}}

\newcommand{\PD}{\operatorname{PD}}

\newcommand{\PT}{\operatorname{PT}}

\newcommand{\SFT}{{\text{SFT}}}

\newcommand{\Order}{\mathcal{O}}

\newcommand{\DD}{{\mathbb D}}

\newcommand{\NN}{{\mathbb N}}
\newcommand{\QQ}{{\mathbb Q}}
\newcommand{\RR}{{\mathbb R}}

\newcommand{\ZZ}{{\mathbb Z}}
\newcommand{\aA}{{\mathcal A}}

\newcommand{\iI}{{\mathcal I}}

\newcommand{\mM}{{\mathcal M}}

\newcommand{\rR}{{\mathcal R}}

\newcommand{\tT}{{\mathcal T}}
\newcommand{\uU}{{\mathcal U}}

\newcommand{\p}{\partial}

\newcommand{\ff}{{\mathfrak f}}
\newcommand{\x}{\times}

\newcommand{\Dsft}{\mathbf{D}_\SFT}

\newcommand{\Surface}{\Sigma}
\newcommand{\RS}{S}

\numberwithin{equation}{section}

\definecolor{blue}{rgb}{0,0,1}
\definecolor{red}{rgb}{1,0,0}
\definecolor{green}{rgb}{0,.7,0}

\definecolor{Chris}{rgb}{0.5,0,1}

\definecolor{Janko}{rgb}{0.6,0.6,0}

\newcommand{\comment}[1]{}

\title[Algebraic Torsion in Contact Manifolds]{Algebraic Torsion in
Contact Manifolds}
\author{Janko Latschev}
\author[Chris Wendl]{Chris Wendl \\ \ \\ 
(with an appendix by Michael Hutchings)}
\address{Janko Latschev \\ 
Fachbereich Mathematik\\
Universit\"at Hamburg\\
Bundesstrasse 55\\
20146 Hamburg\\
Germany}
\email{janko.latschev@math.uni-hamburg.de}
\address{Chris Wendl \\ 
Department of Mathematics \\
University College London \\
Gower Street \\
London WC1E 6BT \\
United Kingdom}
\email{wendl@math.ucl.ac.uk}
\address{Michael Hutchings \\
Mathematics Department \\
970 Evans Hall \\
University of Califronia \\
Berkeley, CA 94720 \\
USA}
\email{hutching@math.berkeley.edu}

\begin{document}

\begin{abstract}
We extract an invariant taking values in $\NN \cup \{\infty\}$,
  which we call the \emph{order of algebraic torsion}, from the Symplectic
  Field Theory of a closed contact manifold, and show that its
  finiteness gives obstructions to the existence of symplectic
  fillings and exact symplectic cobordisms.  A contact 
manifold has algebraic torsion of order~$0$ if and only if it
is algebraically overtwisted (i.e.~has trivial contact homology),
and any contact $3$-manifold with positive Giroux torsion has algebraic
  torsion of order~$1$ (though the converse is not true). 
  We also construct examples for each $k \in
  \NN$ of contact $3$-manifolds that have algebraic torsion of order
  $k$ but not~$k-1$, and derive consequences for contact surgeries on
  such manifolds.\newline 
  The appendix by Michael Hutchings gives an alternative proof of
  our cobordism obstructions in dimension three 
  using a refinement of the contact invariant
  in Embedded Contact Homology.
\end{abstract}

\maketitle

\section{Introduction}

\subsection{Main results}

Symplectic field theory (SFT) is a very general theory of holomorphic
curves in symplectic manifolds which was outlined by Eliashberg,
Givental and Hofer \cite{SFT}, and whose analytical foundations
are currently under development by Hofer, Wysocki and Zehnder, 
cf.~\cite{Hofer:polyfolds}.  It contains as special cases several theories
that have been shown to have powerful consequences in contact 
topology---notably contact homology and Gromov-Witten theory---but 
the more elaborate structure of ``full'' SFT
has yet to find application, as it is usually far too complicated to
compute.  Our goal here is to introduce a numerical invariant, which
we call \emph{algebraic torsion}, that is extracted from the full SFT
algebra and whose finiteness gives obstructions to the existence of
symplectic fillings and exact symplectic cobordisms. 
Algebraic torsion is defined in all dimensions, and we illustrate its
effectiveness by proving explicit nonexistence results for exact
symplectic cobordisms whose ends are certain prescribed nonfillable
contact $3$-manifolds, see Corollary~\ref{cor:noExact} below. To
the best of our knowledge, results of this type are new and seem to be
beyond the present reach of more topologically oriented
methods such as Heegaard Floer homology.

From the point of view taken in this paper, which is 
adapted from \cite{CieliebakLatschev:propaganda} and 
described in more
detail in \S\ref{sec:SFT}, the SFT of a contact manifold $(M,\xi)$
is the homology $H^\SFT_*(M,\xi)$ of a $\ZZ_2$-graded $BV_\infty$-algebra
$(\aA[[\hbar]],\Dsft)$, where $\aA$ has generators $q_\gamma$
for each good closed Reeb orbit~$\gamma$ 
with respect to some nondegenerate contact form
for $\xi$, $\hbar$ is an even variable, and the operator 
$$
\Dsft: \aA[[\hbar]] \to \aA[[\hbar]]
$$
is defined by counting rigid solutions to a suitable abstract 
perturbation of a $J$-holomorphic curve equation in the symplectization
of $(M,\xi)$. The domains for these solutions are punctured
closed Riemann surfaces, and near the punctures the
solutions have so-called positive or negative cylindrical ends. It
follows from the exactness of the symplectic form in the
symplectization that all such curves must have at least
one  positive end. Algebraically, this translates into the fact that
the ground ring 
$\RR[[\hbar]]$ of $\aA$ consists of closed elements with respect to
$\Dsft$. This motivates the following:
\begin{defn}
Let $(M,\xi)$ be a closed manifold of dimension $2n-1$ with a
positive, co-oriented contact structure.  For any integer $k \ge 0$,
we say that $(M,\xi)$ has \emph{algebraic torsion of order~$k$} (or simply
\emph{algebraic $k$-torsion}) if $[\hbar^k] = 0$ in $H^\SFT_*(M,\xi)$.
\end{defn}

Note that although the version of SFT described in
\cite{SFT} has coefficients in the group ring of $H_2(M)$, the homology
$H^\SFT_*(M,\xi)$ above is defined without group ring coefficients---one can
always do this at the cost of reducing the usual $\ZZ$-grading to
a $\ZZ_2$-grading (see \S\ref{sec:SFT} for details).  
We will introduce group ring coefficients later
to obtain a more refined invariant, cf.~Definition~\ref{defn:twistedTorsion}.

In order to state our first main result, we need a few standard
concepts.
Recall that a strong symplectic filling of a contact manifold $(M,\xi)$ is a
compact symplectic manifold $(W,\omega)$ with $\p W = M$ for which there
exists a vector field $Y$, defined near the boundary and
pointing transversely outward there, 
with $\Lie_Y\omega = \omega$ (i.e.~$Y$ is
a \emph{Liouville} vector field) and such that
$\iota_Y\omega|_M$ is a contact form for $\xi$ giving the correct
co-orientation.  More generally, a symplectic cobordism with 
positive end $(M^+,\xi^+)$ and negative end $(M^-,\xi^-)$ is a compact
symplectic manifold $(W,\omega)$ with boundary $M^+ \sqcup (-M^-)$ and a
vector field as above 
with $\xi^\pm = \ker\left( \iota_Y\omega|_{M^\pm}\right)$, 
with the difference that $Y$ is required to point outward
only along $M^+$ and inward along $M^-$.  
Note that
since $\Lie_Y\omega = d(\iota_Y\omega) = \omega$, the symplectic form
is always exact near the boundary of a symplectic cobordism, though
it need not be exact globally.  The flow of $Y$ can be 
used to identify a neighborhood of $\p W$ with 
$$
([0,\epsilon) \x M^-, d(e^s (\iota_Y\omega)|_{M^-})) 
\sqcup ((-\epsilon,0] \x M^+,d(e^s(\iota_Y\omega)|_{M^+})),
$$
and so any symplectic
cobordism in the above sense can be {\em completed} by gluing 
a positive half of the symplectization of $(M^+,\xi^+)$ and a negative
half of the symplectization of $(M^-,\xi^-)$ to the respective
boundaries. Holomorphic curves in completed symplectic cobordisms are
the main object of study in SFT, with the symplectization $\RR \x M$
being an important special case of a completed symplectic cobordism.

A symplectic cobordism $(W,\omega)$ is called \emph{exact} if
the vector field~$Y$ as described above extends globally over~$W$;
equivalently, this means $\omega = d\lambda$ for a $1$-form $\lambda$ on~$W$
whose restrictions to $M^\pm$ define contact forms for $\xi^\pm$.
From the above definition of algebraic torsion
and the general formalism of SFT, we draw the following
  consequence, which is our first main result and is proven in
\S\ref{sec:SFT}.

\begin{thm}
\label{thm:obstructions}
If $(M,\xi)$ has algebraic torsion then it is not strongly fillable.
Moreover, suppose there is an exact symplectic cobordism having
contact manifolds $(M^+,\xi^+)$ and $(M^-,\xi^-)$ as positive and
negative ends respectively: then if $(M^+,\xi^+)$ has algebraic
$k$-torsion, so does $(M^-,\xi^-)$.
\end{thm}

\begin{remark}
\label{remark:disclaimer}
It is time for a more or less standard disclaimer:
All the theorems regarding SFT that we shall state
in this introduction depend on the
analytical foundations of SFT, which remains a large project in progress
by Hofer, Wysocki and Zehnder (see e.g.~\cite{Hofer:polyfolds}).
In particular, the main technical difficulty which is the subject of their
work is to establish a sufficiently well behaved abstract perturbation
scheme so that $H_*^\SFT(M,\xi)$ is well defined and the natural maps
induced by counting solutions to a perturbed holomorphic curve equation
in symplectic cobordisms exist.  We shall take it for granted throughout
the following that such a perturbation scheme exists and has the
properties that its architects claim 
(cf.~Remark~\ref{remark:abstract})---the further details of this scheme 
will be irrelevant to our arguments.  Note however that our main
applications, Corollaries~\ref{cor:noExact} 
and~\ref{cor:noExactGeometric}, can also be proved using the
Embedded Contact Homology techniques described in the appendix
(cf.~Theorem~\ref{thm:ECH}), and
thus do not depend on any unpublished work in progress.
\end{remark}
\begin{remark}
\label{remark:Weinstein}
Algebraic torsion has some obvious applications beyond those that
we will consider in this paper, e.g.~it
is immediate from the formalism of SFT discussed in \S\ref{sec:SFT}
that any contact manifold
with algebraic torsion satisfies the Weinstein conjecture.
\end{remark}

The simplest example of algebraic torsion is the case $k=0$: we will
show in \S\ref{sec:SFT} (Proposition~\ref{prop:algOT}) that this is equivalent
to $(M,\xi)$ having trivial contact homology, in which case it 
is called \emph{algebraically overtwisted},
cf.~\cite{BourgeoisNiederkrueger:algebraically}.
This is the case, for instance, whenever $(M,\xi)$ is an
overtwisted contact $3$-manifold, and in higher dimensions it has been
shown to hold whenever $(M,\xi)$ contains a \emph{plastikstufe}
\cite{BourgeoisNiederkrueger:PS}, or when $(M,\xi)$ is a connected sum with
a certain exotic contact sphere \cite{BourgeoisVanKoert}.

In dimension three, there are also many known examples of contact manifolds
that are tight but not fillable.  An important class of examples is the following:
$(M,\xi)$ is said to have \emph{Giroux torsion} if it admits a contact
embedding of $(T^2 \times [0,1], \xi_T)$ where
$$
\xi_T = \ker \left[ \cos(2\pi t) \ d\theta + \sin(2\pi t)\ d\phi \right]
$$
in coordinates $(\phi,\theta,t) \in T^2 \times [0,1] = S^1 \times S^1 \times
[0,1]$.  
It was shown by D.~Gay \cite{Gay:GirouxTorsion} that contact 
$3$-manifolds with Giroux torsion are never strongly fillable, and a
computation of the twisted Ozsv\'ath-Szab\'o contact invariant due to
Ghiggini and Honda \cite{GhigginiHonda:twisted} shows that
Giroux torsion is also an obstruction to weak fillings whenever
the submanifold $T^2 \times [0,1] \subset M$ separates~$M$.
There are obvious examples of manifolds with these properties
that are also tight.  On $T^3 = S^1 \times S^1 \times S^1$
for example with coordinates $(\phi,\theta,t)$, the contact form
$$
\cos(2\pi N t) \ d\theta + \sin(2\pi N t) \ d\phi
$$
has Giroux torsion for any integer $N \ge 2$, but it also has no contractible
Reeb orbits, which implies that its contact homology cannot vanish.
The original motivation for this project was to find an algebraic 
interpretation of Giroux torsion that implies nonfillability.  
The solution to this problem is the following result, which is implied by
the more general Theorem~\ref{thm:torsion} below:

\begin{thm}
\label{thm:Giroux}
If $(M,\xi)$ is a contact $3$-manifold with Giroux torsion, then it has
algebraic $1$-torsion.
\end{thm}

While it is possible that ``overtwisted'' and
``algebraically overtwisted'' could be equivalent notions in dimension three,
it turns out that the converse of Theorem~\ref{thm:Giroux} is not true.  
We will show this using a special class of contact manifolds
constructed as follows: assume $\Surface_+$ and $\Surface_-$ are compact (not
necessarily connected) oriented surfaces with nonempty diffeomorphic
boundaries, and denote by 
$$
\Surface = \Surface_+ \cup \Surface_-
$$
the closed oriented surface obtained by gluing them along some orientation
reversing diffeomorphism $\p \Surface_+ \to \p \Surface_-$.  We shall assume~$\Surface$ to be
connected.  The common boundary of
$\Surface_\pm$ forms a multicurve $\Gamma \subset \Surface$.  Then by a construction
originally due to Lutz \cite{Lutz:77}, the product $S^1 \times \Surface$
admits a unique (up to isotopy) $S^1$-invariant contact 
structure~$\xi_\Gamma$ for which the loops $S^1 \times \{z\}$ are 
positively/negatively transverse for~$z$ in the interior of $\Surface_\pm$, and
Legendrian for $z \in \Gamma$.  
(We will give a more explicit construction of this contact structure
in \S\ref{sec:higherOrder}.)  By an argument due to Giroux
(see \cite{Massot:vanishing}), $(S^1 \times \Surface,\xi_\Gamma)$ has no
Giroux torsion whenever it has the following two properties:
\begin{itemize}
\item No connected component of~$\Gamma$ is contractible in~$\Surface$,
\item No two connected components of~$\Gamma$ are isotopic in~$\Surface$.
\end{itemize}
It is easy to find examples (see Figure~\ref{fig:disconnected})
for which both these conditions are
satisfied, as well as the assumption in the following result:

\begin{thm}
\label{thm:noGiroux}
If either of $\Surface_+$ or $\Surface_-$ is disconnected, then the $S^1$-invariant
contact manifold $(S^1 \times \Surface,\xi_\Gamma)$ described above has
algebraic $1$-torsion.  In particular, there exist contact $3$-manifolds
that have algebraic $1$-torsion but no Giroux torsion.
\end{thm}

\begin{remark}
Theorem~\ref{thm:obstructions} implies 
that the examples in Theorem~\ref{thm:noGiroux} 
are not strongly fillable.  The latter has been established previously
via vanishing results for the Ozsv\'ath-Szab\'o contact invariant in
sutured Floer homology, see \cites{HondaKazezMatic,Massot:vanishing,
Mathews:torsion}.
\end{remark}

\begin{figure}
\includegraphics{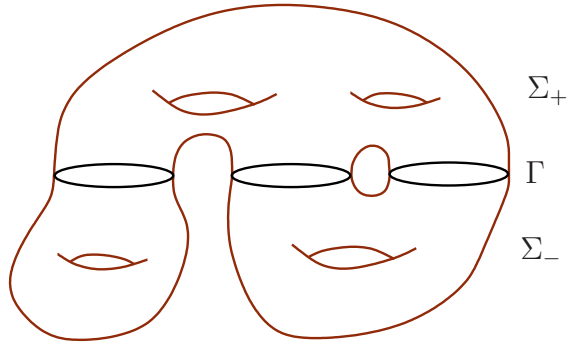}
\caption{\label{fig:disconnected}
A surface $\Surface = \Surface_+ \cup_\Gamma \Surface_-$ such that $(S^1 \times \Surface,\xi_\Gamma)$
has algebraic $1$-torsion but no Giroux torsion.}
\end{figure}

Examples showing that algebraic torsion is interesting for all 
orders can be constructed in almost the same way. In the
construction of $S^1$-invariant contact manifolds $(S^1 \times
\Surface,\xi_\Gamma)$ above, assume that $\Surface_\pm$ are both connected with $k
\geq 1$ boundary components, and that $\Surface_-$ has genus 0 and $\Surface_+$ has
genus $g'>0$. The surface $\Surface$ obtained by gluing will have genus $g=g'+k-1$.
We denote the resulting contact manifold by $(V_g,\xi_k):=(S^1 \times
\Surface,\xi_\Gamma)$.  We then obtain:

\begin{thm}
\label{thm:higherOrder}
$(V_g,\xi_k)$ has algebraic torsion of order $k-1$, but not $k-2$.
\end{thm}

The proof that $(V_g,\xi_k)$ has algebraic torsion of order~$k-1$ 
will be a consequence of Theorem~\ref{thm:torsion} below, which
relates algebraic torsion in dimension 3 to the geometric notion of
\emph{planar torsion} recently introduced by the second
author~\cite{Wendl:openbook2}. This is discussed in detail in
\S\ref{sec:planar}. The proof that there is no algebraic 
torsion of lower order occupies a large part of \S\ref{sec:higherOrder}. It
is based on a combination of algebraic properties of SFT and a
construction of certain explicit contact forms for the contact 
structures $\xi_k$, for which the Reeb dynamics and the holomorphic
curves can be understood sufficiently well.

Combining Theorems~\ref{thm:obstructions} and \ref{thm:higherOrder}
yields the following consequence.

\begin{cor}
\label{cor:noExact}
Suppose $g \ge k \ge 2$.  Then for any exact symplectic cobordism with
negative end $(V_g,\xi_k)$, the positive end does not have algebraic
$(k-2)$-torsion.\newline  
In particular, there exists no exact symplectic 
cobordism with positive end
$(V_{g_+},\xi_{k_+})$ and negative end $(V_{g_-},\xi_{k_-})$ if $k_+<k_-$
(Figure~\ref{fig:Vgk}).
\end{cor}

\begin{figure}
\includegraphics{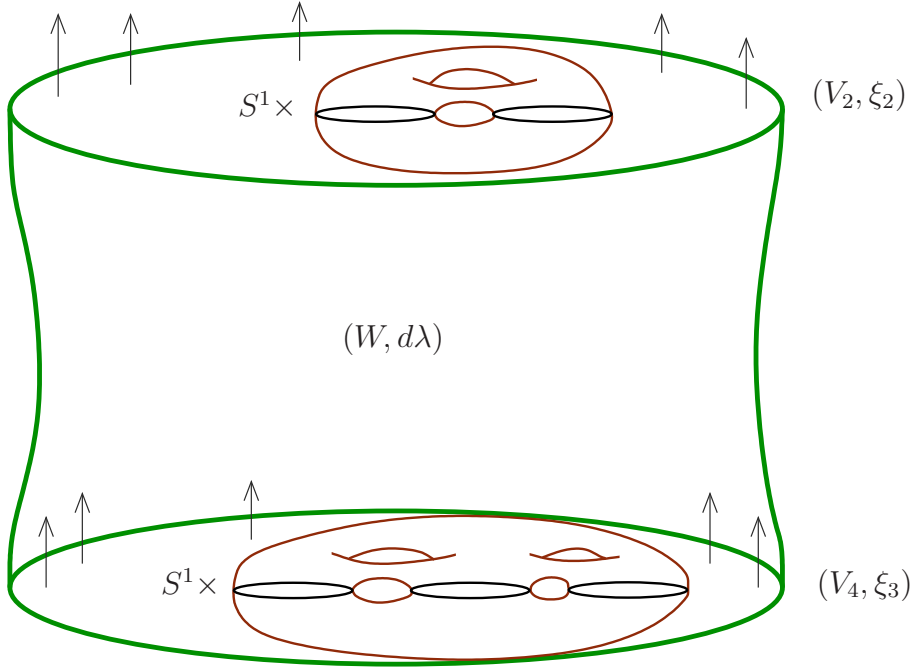}
\caption{\label{fig:Vgk}
An example of an exact symplectic cobordism that cannot
exist according to Corollary~\ref{cor:noExact}.}
\end{figure}

\begin{remark}
\label{remark:nonexact}
The inclusion of the word ``exact'' in the above corollary is crucial,
as a recent construction due to the second author \cite{Wendl:cobordisms}
shows that non-exact symplectic cobordisms exist between \emph{any} two
contact $3$-manifolds with planar torsion.
\end{remark}

\begin{remark}
\label{remark:Jeremy1}
Sometimes exact cobordisms are known to exist
when the negative end has a smaller order of algebraic torsion than the
positive end, e.g.~Etnyre and Honda \cite{EtnyreHonda:cobordisms} have
shown that \emph{any} positive end is allowed if the negative end
is overtwisted (meaning $0$-torsion, in the present context).  Similarly,
Jeremy Van Horn-Morris has explained to us that a Stein cobordism
with negative end $(V_g,\xi_k)$ and positive end $(V_{g+1},\xi_{k+1})$ 
does always exist; cf. Remark~\ref{remark:Jeremy2} in
\S\ref{sec:higherOrder} for an outline of the
construction. Together with Corollary~\ref{cor:noExact}, this 
gives infinite sequences of contact 3-manifolds such that each is
exactly cobordant to its successor, but not vice versa.
\end{remark}

\begin{remark}
\label{remark:Hofer}
The case $k_+ = 1$ of Corollary~\ref{cor:noExact} can be deduced already
from the argument used by Hofer \cite{Hofer:weinstein} to prove the
Weinstein conjecture for overtwisted contact structures.  Indeed,
$(V_{g_+},\xi_{k_+})$ is always overtwisted if $k_+ = 1$, 
and transplanting Hofer's argument from the symplectization to an 
exact symplectic cobordism shows that $(V_{g_-},\xi_{k_-})$ must then
have a contractible Reeb orbit for all nondegenerate contact forms,
which is easily shown to be false if $k_- \ge 2$.  In this sense, the
obstructions coming from algebraic torsion may be seen as a ``higher order''
generalization of Hofer's argument, which incidentally was the starting
point for the developement of SFT.  
\end{remark}

To obtain a more sensitive invariant, we now introduce
a more general notion of algebraic torsion using SFT with
group ring coefficients.  Namely, for any linear subspace
$\rR \subset H_2(M ; \RR)$, one can define the algebra of SFT with
coefficients in the group ring $\RR[H_2(M;\RR) / \rR]$, which means
keeping track of the classes in $H_2(M;\RR) / \rR$ represented by
the holomorphic curves that are counted. We shall denote the
SFT with corresponding coefficients by $H_*^\SFT(M,\xi;\rR)$.
The most important special cases are $\rR = H_2(M;\RR)$ and
$\rR = \{0\}$, called the \emph{untwisted} and \emph{fully twisted}
cases respectively, and $\rR = \ker\Omega$
with~$\Omega$ a closed $2$-form on~$M$.  We shall abbreviate the untwisted
case by $H_*^\SFT(M,\xi) = H_*^\SFT(M,\xi ; H_2(M;\RR))$, and often
write the case $\rR = \ker \Omega$ as
$$
H_*^\SFT(M,\xi,\Omega) := H_*^\SFT(M,\xi ; \ker\Omega).
$$

\begin{defn}
\label{defn:twistedTorsion}
If $(M,\xi)$ is a closed contact manifold, for any integer $k \ge 0$
and closed $2$-form $\Omega$ on~$M$ we say that $(M,\xi)$ has
\emph{$\Omega$}-twisted algebraic $k$-torsion if $[\hbar^k] = 0$ in
$H^\SFT_*(M,\xi,\Omega)$. If this is true for all~$\Omega$, or equivalently, if
$[\hbar^k] = 0$ in $H^\SFT_*(M,\xi ; \{0\})$, then we say that 
$(M,\xi)$ has \emph{fully twisted} algebraic $k$-torsion. 
\end{defn}

To see the significance of algebraic torsion with more general 
coefficients, we consider a more general notion of symplectic fillings,
for which the symplectic form need not be exact near the boundary.

\begin{defn}
\label{defn:stableFilling}
Suppose $(W,\omega)$ is a compact symplectic manifold with boundary
$\p W = M$, and $\xi$ is a positive (with respect to the boundary orientation)
co-oriented contact structure on~$M$.  We call $(W,\omega)$ 
a \emph{stable symplectic
filling} of $(M,\xi)$ if the following conditions are satisfied:
\begin{enumerate}
\item $\omega|_\xi$ is nondegenerate and the induced orientation on~$\xi$
is compatible with its co-orientation
\item $\xi$ admits a nondegenerate contact form~$\lambda$ such that the
Reeb vector field $X_\lambda$ generates the characteristic line 
field on~$\p W$
\item $\xi$ admits a complex bundle structure~$J$ which is tamed by\footnote{The compactness results in \cite{SFTcompactness} are stated for compatible $J$, but they hold without change for tamed $J$ as well.}
both $d\lambda|_\xi$ and $\omega|_\xi$
\end{enumerate}
\end{defn}

Note that a strong filling with Liouville vector field~$Y$ is also a stable 
filling whenever the contact form $\iota_Y\omega|_M$ is nondegenerate, which
can always be assumed after a small perturbation.  In general, the
boundary of a stable filling is a \emph{stable hypersurface} as defined
in \cite{HoferZehnder}, meaning it belongs to a $1$-parameter family of
hypersurfaces in $(W,\omega)$ whose Hamiltonian dynamics are all conjugate.
In particular, the pair $(\lambda,\omega|_M)$ defines a
\emph{stable Hamiltonian structure} on~$M$ 
(cf.~\cite{CieliebakVolkov}).
\begin{thm}
\label{thm:stable}
If $(M,\xi)$ is a closed contact manifold with $\Omega$-twisted 
algebraic torsion for some closed $2$-form~$\Omega$ on~$M$, then
it does not admit any stable filling $(W,\omega)$ for which
$\omega|_{M}$ is cohomologous to~$\Omega$.  In particular, if
$(M,\xi)$ has fully twisted algebraic torsion, then it is not
stably fillable.
\end{thm}

Recall that for $\dim M = 3$, $(W,\omega)$ with $\p W = M$ is said to be
a \emph{weak symplectic filling} of $(M,\xi)$ if $\omega|_\xi > 0$.
Thus a stable filling is also a weak filling.  What's far less obvious
is that the converse is true up to deformation: by
\cite{NiederkruegerWendl}*{Theorem~2.8}, every weak filling can be
deformed near its boundary 
to a stable filling of the same contact manifold, hence
weak and stable fillability are completely equivalent notions in
dimension three. Theorem~\ref{thm:stable} thus implies:

\begin{cor}
\label{cor:weak}
Contact $3$-manifolds with fully twisted algebraic torsion are not
weakly fillable.
\end{cor}
Figure~\ref{fig:torsionDomains} in \S~\ref{sec:planar} below
shows some examples to which this result applies, including one that
has no Giroux torsion; see also Theorem~\ref{thm:torsion} below, and
\cite{NiederkruegerWendl}.

In higher dimensions, it is not hard to find examples
of stable fillings for which the symplectic form is not exact near
the boundary, though it's less obvious whether there are also
examples which are not strongly fillable.  Such
examples are found in the work in progress by Massot, Niederkr\"uger and
the second author \cite{MassotNiederkruegerWendl}, which defines a suitable
generalization of weak fillings to arbitrary dimensions: in a nutshell, 
$(W,\omega)$ with $\p W = M$ is a weak filling of $(M,\xi)$ if~$\omega$
tames an almost complex structure~$J$ that preserves~$\xi$ and is also
tamed by the natural conformal symplectic structure on~$\xi$.  Under this
definition, one can use an existence result of Cieliebak-Volkov 
\cite{CieliebakVolkov} to show that weak and stable fillability are equivalent,
see \cite{MassotNiederkruegerWendl} for details.
Thus SFT also gives obstructions to weak filling in all
dimensions, where the distinction between ``strong'' and ``weak'' is
detected algebraically via the choice of coefficients.

As already mentioned, the second author \cite{Wendl:openbook2}
recently introduced
a new class of filling obstructions in dimension three called \emph{planar
torsion}, which also has a nonnegative integer-valued \emph{order}.
A contact $3$-manifold is then overtwisted if and only if it has planar 
$0$-torsion, and Giroux torsion implies planar $1$-torsion.
We will recall the definition of planar torsion and $\Omega$-separating 
planar torsion in \S\ref{sec:planar}, and prove the following generalization 
of Theorem~\ref{thm:Giroux}. 

\begin{thm}
\label{thm:torsion}
Suppose $(M,\xi)$ is a closed contact $3$-manifold, $\Omega$ is
a closed $2$-form on~$M$ and $k \ge 0$ is an integer.
\begin{enumerate}
\item If $(M,\xi)$ has planar $k$-torsion then it also has algebraic 
$k$-torsion.
\item If $(M,\xi)$ has $\Omega$-separating planar $k$-torsion then it
also has $\Omega$-twisted algebraic $k$-torsion.
\end{enumerate}
\end{thm}

\begin{remark}
Together with Theorem~\ref{thm:obstructions} and 
Corollary~\ref{cor:weak}, this yields new proofs that contact $3$-manifolds
with planar torsion are not strongly fillable, and also not weakly
fillable if the planar torsion is fully separating.  These two results were
first proved in \cite{Wendl:openbook2} and \cite{NiederkruegerWendl}
respectively.  The former also proves a vanishing result for the ECH
contact invariant which is closely analogous to Theorem~\ref{thm:torsion}
and has thus far been inaccessible from the direction of Heegaard Floer
homology.  Our argument in fact implies a refinement of this vanishing
result in terms of the relative filtration on ECH introduced
in the appendix; see Theorem~\ref{thm:ECH} below.
\end{remark}

We can now state a more geometric analogue of
Corollary~\ref{cor:noExact}.  The notion of planar torsion gives
rise to a contact invariant $\PT(M,\xi) \in \NN \cup \{0,\infty\}$,
the \emph{minimal order} of planar torsion, defined by
$$
\PT(M,\xi) := \sup \left\{ k \ge 0 \ \big|\ 
\text{$(M,\xi)$ has no planar $\ell$-torsion for any $\ell < k$} \right\}.
$$
This number is infinite whenever $(M,\xi)$ is strongly fillable, and is
positive if and only if $(M,\xi)$ is tight.
Recall that contact connected sums and $(-1)$-surgeries always yield
Stein cobordisms between contact $3$-manifolds 
(see e.g.~\cite{Geiges:book}).
The following can then be thought of as demonstrating 
a higher order
variant of the well known conjecture that such surgeries always
\emph{preserve tightness}.

\begin{cor}
\label{cor:noExactGeometric}
For any $g \ge k \ge 1$, $\PT(V_g,\xi_k) = k-1$.  Moreover, suppose
$(M,\xi)$ is any contact $3$-manifold that 
can be obtained from $(V_g,\xi_k)$ by a sequence of
\begin{itemize}
\item contact connected sums with itself or
exactly fillable contact manifolds, and/or
\item contact $(-1)$-surgeries.  
\end{itemize}
Then $\PT(M,\xi) \ge k-1$.
\end{cor}

At present, we do not know any example for which the minimal
order of algebraic torsion is strictly smaller than the minimal order
of planar torsion, but Theorem~\ref{thm:noGiroux}
seems to suggest that such examples are likely to exist.

Here is a summary of the remainder of the
paper.  
In \S\ref{sec:SFT} we review the algebraic formalism of SFT as a
$BV_\infty$-algebra, in particular proving Theorems~\ref{thm:obstructions}
and~\ref{thm:stable}.
In \S\ref{sec:planar} we review the definition of planar torsion and
prove Theorem~\ref{thm:torsion}, as an easy application of some results
on holomorphic curves from \cite{Wendl:openbook2}.  The $S^1$-invariant
examples $(S^1 \times \Surface, \xi_\Gamma)$ are then treated at length in
\S\ref{sec:higherOrder}, leading to the proofs of 
Theorems~\ref{thm:noGiroux} and~\ref{thm:higherOrder}. 
We close with a brief
discussion of open questions and related issues in \S\ref{sec:open}.

\smallskip

In Michael Hutchings's appendix to this paper, it is shown that
the applications to $3$-dimensional contact topology described above
can also be proved using methods from Embedded Contact Homology.
Indeed, as remarked above, all of our examples of contact $3$-manifolds
with algebraic torsion can also be shown to have vanishing ECH
contact invariant, suggesting that a refinement of the latter should exist
which could detect the order of torsion. 
The appendix carries out enough of this program to suffice for 
our applications.  In particular, Hutchings associates to any
closed contact $3$-manifold $(M,\xi)$ with generic contact form~$\lambda$,
compatible complex structure~$J$ and positive number 
$T \in (0,\infty]$, two nonnegative (possibly infinite) integers
$f^T(M,\lambda,J)$ and $f^T_{\text{simp}}(M,\lambda,J)$.  These can be finite
only if the ECH contact invariant vanishes, and they
have the property that
$$
f_{\text{simp}}^{T_+}(M^+,\lambda^+,J^+) \ge f^{T_-}(M^-,\lambda^-,J^-)
$$
whenever there is an exact cobordism $(X,d\lambda)$ with
$\lambda = e^s\lambda^\pm$ at the positive/negative end and
$T_- \ge T_+$ (cf.~Theorem~\ref{thm:fm}).
Since $f^T$ and $f_{\text{simp}}^T$ are defined by counting embedded
holomorphic curves in symplectizations, our SFT computations can be
reinterpreted as estimates of these integers,
leading to the following:

\begin{thm}
\label{thm:ECH}
\ \\

\vspace{-12pt}
\begin{enumerate}
\item If $(M,\xi)$ has planar $k$-torsion, then~$\xi$ admits a nondegenerate
contact form~$\lambda$ and generic complex structure~$J$ such that
$f_{\text{simp}}^\infty(M,\lambda,J) \le k$.
\item For any $g \ge k \ge 1$, $(V_g,\xi_k)$ admits a sequence of
generic contact forms and complex structures $(\lambda_i,J_i)$ such that:
\begin{enumerate}
\item $f^{T_i}(V_g,\lambda_i,J_i) \ge k - 1$ for some sequence of real
numbers $T_i \to +\infty$,
\item For $i < j$, there is an exact symplectic cobordism $(X,d\lambda)$
such that $\lambda$ matches $e^s\lambda_i$ at the positive end and
$e^s\lambda_j$ at the negative end.
\end{enumerate}
\end{enumerate}
\end{thm}
As mentioned in Remark~\ref{remark:disclaimer} above, 
this immediately implies an alternative proof of 
Corollaries~\ref{cor:noExact} and~\ref{cor:noExactGeometric},
cf.~Corollary~\ref{cor:ECHobstructions} in the appendix.

\subsection*{Acknowledgments}
The authors would like to thank Kai Cieliebak, Helmut Hofer, 
Patrick Massot, Klaus Niederkr\"uger and Jeremy Van Horn-Morris for
helpful conversations.  
The inclusion of Michael Hutchings's appendix came about due to
discussions between Hutchings and the second author at the MSRI Workshop
\textsl{Symplectic and Contact Topology and Dynamics: Puzzles and Horizons}
in March 2010.

CW gratefully acknowledges support from an Alexander von Humboldt Foundation
research fellowship.  This work was started when both authors
held positions at ETH Z\"urich, and significant progress was made
during a joint conference visit to the Lorentz Center in Leiden. It is
a pleasure to thank these institutions for the stimulating working
environment.

\section{Review of SFT as a $BV_\infty$-algebra}
\label{sec:SFT}

The general framework of SFT, in particular its algebraic structure,
was laid out in \cite{SFT} (see also \cite{Eliashberg:SFT} for a more
recent point of view), whereas the analytic foundations are the subject
of ongoing work by Hofer-Wysocki-Zehnder (see \cite{Hofer:polyfolds}).
In this section, we will take the existence of SFT as described in
\cite{SFT} for granted and review a version of the theory which is
readily derived from this description
(cf. \cite{CieliebakLatschev:propaganda} for some details of this
translation). To keep the discussion reasonably brief, we will
frequently refer to these sources for
details. Theorems~\ref{thm:obstructions} and~\ref{thm:stable}
will be simple consequences of the algebraic properties of SFT.

\subsection{Review of the basic setup of SFT}
Let $(M,\xi)$ be a closed manifold of dimension $2n-1$ with a
co-oriented contact structure. To describe SFT, one needs to fix a
nondegenerate contact form $\lambda$, as well as some additional
choices, which we denote by a single letter $\ff$ (for framing). The
most important of these are: a cylindrical almost complex structure
$J$ on the symplectization of $M$, coherent orientations for
the moduli space of finite energy 
$J$-holomorphic curves, an abstract perturbation scheme for the $J$-holomorphic
curve equation and suitable spanning surfaces for Reeb orbits. 

Given a linear subspace $\rR \subset H_2(M ; \RR)$, let
$R_\rR:=\RR[H_2(M;\RR) / \rR]$ denote the group ring over $\RR$ of
$H_2(M;\RR) / \rR$, whose elements we write as $\sum a_i z^{d_i}$
with $a_i \in \RR$ and $d_i \in H_2(M ; \RR) / \rR$. 
Define $\aA= \aA(\lambda)$ to be the $\ZZ_2$-graded algebra with unit
over the group ring $R_\rR$, generated by variables
$q_\gamma$, where $\gamma$ ranges over the collection of good closed
Reeb orbits for $\lambda$ (cf. \cite[footnote on p.~566 and Remarks
1.9.2 and 1.9.6]{SFT}), and the degree of $q_\gamma$ is defined as
$$
|q_\gamma| := n-3 + \muCZ(\gamma) \mod 2.
$$
Here $\muCZ(\gamma)$ denotes the mod 2 Conley-Zehnder index of the closed
orbit $\gamma$, which is defined in terms of the
linearized Poincare return map for $\gamma$
(cf. \cite[p.~567]{SFT}). We also introduce an extra variable $\hbar$
of even degree and consider the algebra of formal power series
$\aA[[\hbar]]$.

To construct the differential, one chooses a cylindrical almost
complex structure $J$ on the symplectization $(\RR \x M,
\omega=d(e^s\lambda))$.  To be precise, we say that an 
almost complex structure~$J$ on $\RR \times M$ is \emph{adapted
to~$\lambda$} if it is $\RR$-invariant, maps the unit vector $\p_s$
in the $\RR$-direction to the Reeb vector field $X_\lambda$ of~$\lambda$,
and restricts to a tamed complex structure on the symplectic
vector bundle $(\xi,d\lambda)$.  After a choice of spanning surfaces as
in \cite[p.~566, see also p.~651]{SFT}, the projection to $M$ of each
finite energy holomorphic curve~$u$ can be capped off to a
$2$-cycle in $M$, 
and so it gives rise to a homology class in~$H_2(M)$, 
which we project to define $[u] \in H_2(M ; \RR) / \rR$.

As explained in \cite[section 6]{CieliebakLatschev:propaganda}, the
count of suitably perturbed $J$-holomorphic curves in $\RR \x M$ with
finite Hofer energy gives rise to a differential operator
$$
\Dsft: \aA[[\hbar]] \to \aA[[\hbar]]
$$
such that
\begin{itemize}
\item $\Dsft$ is odd and squares to zero,
\item $\Dsft(1)=0$, and
\item $\Dsft = \sum_{k \geq 1} D_k \hbar^{k-1}$, where $D_k:\aA \to
  \aA$ is a differential operator of order $\leq k$.
\end{itemize}
More precisely, 
$$
D_k= \sum_{\stackrel{\Gamma_+,\Gamma_-,g,d}{|\Gamma_+| + g = k}}
n_g(\Gamma_-, \Gamma_+,d)
\frac{1}{C(\Gamma_-,\Gamma_+)} q_{\gamma_1^-}\cdots q_{\gamma_{s_-}^-} z^d \frac
\p {\p q_{\gamma_1^+}} \cdots \frac \p {\p q_{\gamma_{s_+}^+}},
$$
where the sum ranges over all nonnegative integers $g \ge 0$, homology classes
$d \in H_2(M;\RR) / \rR$ and ordered (possibly empty) 
collections of good closed Reeb orbits
$\Gamma_\pm = (\gamma^\pm_1,\ldots,\gamma^\pm_{s_\pm})$ such that
$s_+ + g = k$.  The number
$n_g(\Gamma_-, \Gamma_+,d)\in \QQ$ denotes the count of (suitably perturbed)
holomorphic curves of genus $g$ with positive asymptotics $\Gamma_+$ 
and negative asymptotics $\Gamma_-$ in the homology class~$d$, including
asymptotic markers as explained in \cite[p.~622f]{SFT}.  Finally, 
$C(\Gamma_-,\Gamma_+) \in \NN$ is a combinatorial factor defined as
$$
C(\Gamma_-,\Gamma_+) =
s_-!s_+!\kappa_{\gamma_1^-}\cdots
  \kappa_{\gamma_{s_-}^-} \kappa_{\gamma_1^+}\cdots
  \kappa_{\gamma_{s_+}^+},
$$
where $\kappa_\gamma$ denotes the covering multiplicity of the Reeb
orbit $\gamma$. 

Observe in particular that for $Q=q_{\gamma_1}\cdots q_{\gamma_r}$, 
the constant coefficient
(i.e. the element of the ground ring) in $D_k(Q)$ for $k \geq r$
corresponds to the count of holomorphic curves of genus $k-r$ with
positive asymptotics $\Gamma=\{\gamma_1, \cdots,\gamma_r\}$ and no
negative ends.   

The homology of $(\aA[[\hbar]],\Dsft)$ is denoted by
$H_*^\SFT(M, \lambda,\ff ; \rR)$.
Note that by definition the operator
$\Dsft$ commutes with $\hbar$ and with elements of~$R_\rR$. 
As $\Dsft$ is not a derivation, the
homology is not an algebra, but only an $R_\rR[[\hbar]]$-module. However,
the element $1 \in \aA$ and all its $R_\rR[[\hbar]]$-multiples
are always closed by the second property above, and so they define
preferred homology classes.  The special case $\rR = H_2(M;\RR)$
is of particular importance: then $R_\rR$ reduces to the trivial group
ring~$\RR$ and we abbreviate
$$
H_*^\SFT(M, \lambda,\ff) := H_*^\SFT(M, \lambda,\ff ; H_2(M;\RR)),
$$
which we refer to as the SFT with \emph{untwisted} coefficients.
Similarly, for any closed $2$-form $\Omega$ on~$M$, we abbreviate
the special case $\rR = \ker\Omega \subset H_2(M;\RR)$ by
$$
H_*^\SFT(M, \lambda,\ff,\Omega) := H_*^\SFT(M, \lambda,\ff ; \ker\Omega)
$$
and call this the SFT with \emph{$\Omega$-twisted} coefficients.
The \emph{fully} twisted SFT is 
$$
H_*^\SFT(M, \lambda,\ff ; \{0\}),
$$
defined by taking~$\rR$ to
be the trivial subspace.  Observe that the inclusions
$\{0\} \hookrightarrow \ker\Omega \hookrightarrow H_2(M;\RR)$ induce natural
$\RR[[\hbar]]$-module morphisms
$$
H_*^\SFT(M, \lambda,\ff ; \{0\}) \to
H_*^\SFT(M, \lambda,\ff,\Omega) \to
H_*^\SFT(M, \lambda,\ff).
$$

A framed cobordism $(X,\omega,\ff_X)$ with positive end
$(M^+,\lambda^+,\ff^+)$ and negative end $(M^-,\lambda^-,\ff^-)$ 
is a symplectic cobordism $(X,\omega)$ with oriented boundary
$M^+ \sqcup (-M^-)$, together with the following additional data: 
\begin{itemize}
\item a Liouville vector field $Y$, defined near the boundary,
  pointing outward at $M^+$ and inward at $M^-$, such that
  $\iota_Y\omega|_{M^\pm}=\lambda^\pm$, 
\item a tamed almost complex structure $J$ interpolating between
  the given cylindrical structures $J^\pm$ at the ends,
\item coherent orientations for the moduli spaces of finite energy
$J$-holomorphic curves in the completion of~$X$,
\item an abstract perturbation scheme compatible with~$\ff^+$
and~$\ff^-$, and
\item spanning surfaces for the cobordism as described in
    \cite[p.~571f]{SFT}.
\end{itemize}
As explained in \cite[section 8]{CieliebakLatschev:propaganda}, such a
cobordism gives rise to a morphism from 
$H_*^\SFT(M^+,\lambda^+,\ff^+)$ to $H_*^\SFT(M^-,\lambda^-,\ff^-)$
after suitably twisting the differential as follows.  

Suppose $\rR^\pm \subset H_2(M^\pm;\RR)$ and
$\rR(X) \subset \ker\omega \subset H_2(X;\RR)$ are linear subspaces
such that the maps $H_2(M^\pm;\RR) \to H_2(X;\RR)$ induced by the
inclusions $M^\pm \hookrightarrow X$ map $\rR^\pm$ into~$\rR(X)$.
Define the group rings $R_{\rR^\pm} = \RR[H_2(M;\RR) / \rR^\pm]$ and
$R_{\rR(X)}=\RR[H_2(X;\RR) / \rR(X)]$, and let
$(\aA^\pm[[\hbar]],\Dsft^\pm)$ denote the $BV_\infty$-algebras as
defined above for $(M^\pm,\lambda^\pm,\ff^\pm)$ with coefficients in
$R_{\rR^\pm}$.  We also denote by
$\aA^-_X$ the algebra generated  by the $q_\gamma^-$
with coefficients in $R_{\rR(X)}$ instead of $R_{\rR^-}$, 
Novikov completed as described in \cite[p.624]{SFT} (note that
integration of $\omega$ gives a well defined homomorphism
$H_2(X;\RR)/\rR(X) \to \RR$).
The inclusions $M^\pm \hookrightarrow X$ give rise to 
morphisms $H_2(M^\pm;\RR) / \rR^\pm \to H_2(X;\RR) / \rR(X)$ and
$R_{\rR^\pm} \to R_{\rR(X)}$, which in particular
determine a morphism of algebras $\aA^- \to \aA_X^-$.

Now $(X,\omega,\ff_X)$ gives rise to several structures, the first of
which is an element $A\in \hbar^{-1}\aA^-_X[[\hbar]]$ satisfying
$\Dsft^-(e^A)=0$, which is obtained from counting holomorphic
curves in $X$ with no positive punctures (these may exist only if $X$
is not exact). Using this, one can define a twisted differential
$\mathbf{D}^-_X: \aA^-_X[[\hbar]] \to \aA^-_X[[\hbar]]$ 
by the formula
$$
\mathbf{D}^-_X(Q)= e^{-A}\Dsft^-(e^A\cdot Q).
$$
In this way, we get a twisted version of SFT for
$(M^-,\lambda^-,\ff^-)$, which depends on $(X,\omega,\ff_X)$.  

\begin{remark}
Above we have defined two
kinds of twisted versions of SFT, namely SFT twisted with respect to a closed
two-form, and the twisted SFT of the negative end of a (non-exact)
symplectic cobordism. We hope that it is always clear from the context
which kind of twisting is meant.
\end{remark}

The other structure one obtains is a chain map
$\Phi=e^\phi:(\aA^+[[\hbar]],\Dsft^+) \to
(\aA^-_X[[\hbar]],\mathbf{D}_X^-)$ determined by a map
$\phi=\phi_X:\aA^+ \to \aA^-_X[[\hbar]]$ satisfying
\begin{itemize}
\item $\phi$ is even and $\phi(1)=0$,
\item $e^{\phi}\Dsft^+ = \mathbf{D}^-_X e^{\phi}$, and
\item $\phi= \sum_{k \geq 1} \phi_k \hbar^{k-1}$, where each
  $\phi_k:\aA^+ \to \aA^-_X$ is a differential operator of order
  $\leq k$ over the zero morphism.\footnote{Given a
    morphism $\rho:A_1 \to A_2$ between graded commutative
    algebras, a homogeneous linear map $D:A_1 \to A_2$ is a
    differential operator of order $\leq k$ over $\rho$ if for each
    homogeneous element $a\in A_1$ the map $x \mapsto
    D(ax)-(-1)^{|D||a|}\rho(a)D(x)$ is a differential operator of order
    $\leq k-1$, with the convention that the zero map has order $\leq -1$.}
\end{itemize}

This $\phi$ counts holomorphic curves in $X$ with at least one
positive puncture. The first condition above translates to the fact
that $\Phi(1)=1$. Again $\Phi$ is $\hbar$-linear, so it induces a
morphism of $\RR[[\hbar]]$-modules $H_*(\aA^+,\Dsft^+) \to
H_*(\aA^-_X,\mathbf{D}^-_X)$, which maps the preferred class 
$[1]\in H_*(\aA^+,\Dsft^+)$ and its $R_{M^+}[[\hbar]]$-multiples to the
corresponding classes in $H_*(\aA^-_X,\mathbf{D}^-_X)$.

To discuss the invariance properties of SFT, one studies holomorphic
curves in topologically trivial cobordisms $\RR \times M$. More
precisely, given two contact forms $\lambda^\pm$ for the same contact
structure~$\xi$, there is a constant $c>0$ and an exact symplectic form
$\omega = d(e^s \lambda_{s})$ on $\RR \times M$ such that the primitive
$\lambda_{s}$ agrees with $c \lambda^-$ at the negative end and with
$\lambda^+$ at the positive end of the cobordism. Similarly, one finds
a framing $\ff_{\RR \times M}$ compatible with given framings $\ff^\pm$
at the ends. Note that in this case $\ker\omega = 
H_2(X) = H_2(M)$, so we can choose $\rR^\pm = \rR =\rR(X)$
and observe that the completion process in the definition of
  $\aA^-_X$ is trivial since $\omega$ is exact,  
giving rise to a natural identification of $\aA^-_X$ with $\aA^-$.
Likewise, $A\in \hbar^{-1}\aA^-$ vanishes as the cobordism is
exact. Since rescaling of $\lambda$ does not influence the count of
holomorphic curves, we obtain a chain map $(\aA^+[[\hbar]],\Dsft^+) \to
(\aA^-[[\hbar]],\Dsft^-)$.  

Reversing the roles of $\lambda^+$ and $\lambda^-$, one obtains
a similar chain map in the other direction, 
and a deformation argument implies that
both compositions are chain homotopic to the identity maps on
$(\aA^\pm,\Dsft^\pm)$, respectively.  In particular, they induce
$R_\rR[[\hbar]]$-module isomorphisms on
homology, so that the contact invariant
$$
H_*^\SFT(M, \xi ; \rR) := H_*^\SFT(M, \lambda,\ff ; \rR)
$$
is well defined up to natural isomorphisms.  
It is important for us to observe that,
by construction, these morphisms are the identity on $R_\rR[[\hbar]]
\subset \aA^\pm$, thus $H^\SFT_*(M,\xi ; \rR)$
comes with preferred homology classes
associated to the elements of $R_\rR[[\hbar]]$.  Considering
the special cases where $\rR$ is~$\{0\}$, $\ker\Omega$ or $H_2(M;\RR)$ again
gives rise to the fully twisted, $\Omega$-twisted and untwisted versions
respectively, with natural $\RR[[\hbar]]$-module morphisms
\begin{equation}
\label{eqn:coefficients}
H_*^\SFT(M, \xi ; \{0\}) \to
H_*^\SFT(M, \xi,\Omega) \to
H_*^\SFT(M, \xi).
\end{equation}

\begin{remark}
\label{rem:truncation}
The above discussion of morphisms can be refined slightly as follows.
Given a nondegenerate contact form $\lambda$ and a constant $T>0$, we
can consider the linear subspace $\aA(\lambda, T) \subset \aA(\lambda)$ in the
corresponding chain level algebra generated by all the monomials of the form
$q_{\gamma_1}\dots q_{\gamma_r}$ for which the total action is bounded
by $T$, i.e.
$$
\sum_{j=1}^r \int_{\gamma_j} \lambda < T.
$$
Since the energy of holomorphic curves contributing to $\Dsft$ is
nonnegative and given by the action difference of the asymptotics, the
operator $\Dsft$ restricts to define a differential
$$
\Dsft : \aA(\lambda,T)[[\hbar]] \to \aA(\lambda,T)[[\hbar]].
$$
Moreover, if $\omega = d(e^s \lambda_s)$ is a symplectic form
on $\RR\times M$ such that $\lambda$ agrees with $\lambda^+$ at the
positive end and $c\lambda^-$ at the negative end, then the resulting
morphism respects the truncation with suitable rescaling, i.e.~it gives
rise to a chain map
$$
\Phi_T: (\aA(\lambda^+,T)[[\hbar]],\Dsft^+) \to
(\aA(c\lambda^-,T)[[\hbar]],\Dsft^-) 
= (\aA(\lambda^-,T / c)[[\hbar]],\Dsft^-).
$$
Beware however that, due to the rescaling of forms for the cylindrical
cobordisms, there is no meaningful filtration on 
$H^\SFT_*(M,\xi ; \rR)$.

In the proof of Theorem~\ref{thm:higherOrder} we will use this refinement
in the situation where $\lambda^-$ has only its periodic orbits of
action at most $T$ nondegenerate, in which case the truncated complex
$(\aA(\lambda^-,T)[[\hbar]],\Dsft^-)$ can still be constructed with all the
required properties. 
\end{remark}

It is useful to consider how the chain map
$\Phi : (\aA^+[[\hbar]],\Dsft^+) \to (\aA^-_X[[\hbar]],\mathbf{D}_X^-)$
induced by a symplectic cobordism~$(X,\omega)$ simplifies whenever certain 
natural extra assumptions are placed on~$X$.  First, suppose that
$(X,\omega)$ is an \emph{exact} cobordism.  As we already observed above,
in this case~$X$ contains no holomorphic curves without positive ends,
hence the ``twisting'' term $A\in \hbar^{-1}\aA^-_X[[\hbar]]$ vanishes.
Moreover, since $\ker\omega = H_2(X;\RR)$, we can set $\rR(X) = H_2(X;\RR)$
and reduce $R_{\rR(X)}$ to the untwisted coefficient ring~$\RR$.
Making corresponding choices $\rR^\pm = H_2(M^\pm;\RR)$ so that
$R_{\rR^\pm} = \RR$ for the positive and negative ends, we then have a
natural identification of the two chain complexes
$(\aA^-_X[[\hbar]],\mathbf{D}_X^-)$ and $(\aA^-[[\hbar]],\Dsft^-)$,
hence the aforementioned chain map yields the following:

\begin{prop}
\label{prop:exactCobordisms}
Any exact symplectic cobordism $(X,\omega)$ with positive end
$(M^+,\xi^+)$ and negative end $(M^-,\xi^-)$ gives rise to a natural
$\RR[[\hbar]]$-module morphism on the untwisted SFT,
$$
\Phi_X : H_*^\SFT(M^+,\xi^+) \to H_*^\SFT(M^-,\xi^-).
$$
\end{prop}

Now suppose $(X,\omega)$ is a strong filling of $(M^+,\xi^+)$, which we may
view as a symplectic cobordism whose negative end $(M^-,\xi^-)$ 
is the empty set.
For any given subspace $\rR(X) \subset \ker\omega$,
the Novikov completion $\overline{R_{\rR(X)}}$ of
$R_{\rR(X)}$ need not be trivial, but the
chain complex $(\aA^-_X[[\hbar]],\mathbf{D}_X^-)$ has no generators other
than the unit, and its differential vanishes, hence its homology is
simply~$\overline{R_{\rR(X)}}[[\hbar]]$.  Choosing $\rR \subset
H_2(M;\RR)$ so that the natural map $H_2(M;\RR) \to H_2(X;\RR)$ induced by
the inclusion $M \hookrightarrow X$ takes $\rR$ into $\rR(X)$, we also
obtain a natural $\RR[[\hbar]]$-module morphism 
$R_\rR[[\hbar]] \to R_{\rR(X)}[[\hbar]]$.  Note that since~$\omega$ is
necessarily  exact near $\p X$, we can always choose $\rR(X) =
\ker\omega$ and $\rR = H_2(M;\RR)$.  We obtain:

\begin{prop}
\label{prop:strongFillings}
Suppose $(X,\omega)$ is a strong filling of $(M,\xi)$, and
$\rR(X) \subset \ker\omega \subset H_2(X;\RR)$ and $\rR \subset H_2(M;\RR)$
are linear subspaces for which the natural map from $H_2(M;\RR)$ to
$H_2(X;\RR)$ takes~$\rR$ into $\rR(X)$.  Then there is a
natural $\RR[[\hbar]]$-module morphism
$$
\Phi_X : H_*^\SFT(M,\xi ; \rR) \to \overline{R_{\rR(X)}}[[\hbar]],
$$
which acts on $R_\rR[[\hbar]] \subset H_*^\SFT(M,\xi;\rR)$ as the
natural map to $R_{\rR(X)}[[\hbar]]$ 
induced by the inclusion $M \hookrightarrow X$.
In particular, the untwisted SFT of $(M,\xi)$ admits an $\RR[[\hbar]]$-module
morphism
$$
\Phi_X : H_*^\SFT(M,\xi) \to \overline{R_{\ker\omega}}[[\hbar]].
$$
\end{prop}

Finally, we generalize the above to allow for stable symplectic fillings
as defined in the introduction.  Recall that if $(X,\omega)$ is a stable
filling of $(M,\xi)$ and we write $\Omega := \omega|_M$, 
then~$\xi$ admits a nondegenerate contact form~$\lambda$
and complex structure~$J_\xi$ such that $\omega|_\xi$ and $d\lambda|_\xi$
both define symplectic bundle structures taming~$J_\xi$, and
the Reeb vector field $X_\lambda$ generates $\ker\Omega$.  In particular,
the pair $(\lambda,\Omega)$ is then a stable Hamiltonian structure,
meaning it satisfies:
\begin{enumerate}
\item $\lambda \wedge \Omega^{n-1} > 0$,
\item $d\Omega = 0$,
\item $\ker\Omega \subset \ker d\lambda$.
\end{enumerate}
A routine Moser deformation argument shows that a neighborhood of~$\p X$
in $(X,\omega)$ can then be identified symplectically with the collar
$$
((-\epsilon,0] \times M, d(t\lambda) + \Omega)
$$
for $\epsilon > 0$ sufficiently small.  Choose a small number
$\epsilon_0 > 0$ and define
$$
\tT := \{ \varphi \in C^\infty([0,\infty) \to [0,\epsilon_0))\ |\ 
\text{$\varphi' > 0$ everywhere and $\varphi(t) = t$ near $t=0$} \}.
$$
Then if $\epsilon_0$ is small enough, every $\varphi \in \tT$ gives rise
to a symplectic form $\omega_\varphi$ on the completion
$\widehat{X} := X \cup_M \left( [0,\infty) \times M \right)$, defined by
$$
\omega_\varphi =
\begin{cases}
\omega & \text{ on~$X$},\\
d\left( \varphi(t)\lambda\right) + \Omega & \text{ on $[0,\infty) \times M$}.
\end{cases}
$$
Define a cylindrical almost complex structure on $[0,\infty) \times M$
which maps $\p_s$ to $X_\lambda$ and restricts to $J_\xi$ on~$\xi$; due
to the compatibility assumptions on $J_\xi$, this is 
$\omega_\varphi$-tame for all possible choices of $\varphi \in \tT$.
We can thus extend it to a generic $\omega_\varphi$-tame almost complex
structure~$J$ on~$\widehat{X}$.  Then one can generalize the previous
discussion by considering punctured $J$-holomorphic curves
$u : \dot{\RS} \to \widehat{X}$ that satisfy the finite energy
condition
$$
E(u) := \sup_{\varphi \in \tT} \int_{\dot{\RS}} u^* \omega_\varphi.
$$
This definition of energy is equivalent to the one given in 
\cite{SFTcompactness} in the sense that bounds on either imply bounds on
the other; it follows that the compactness theorems of
\cite{SFTcompactness} apply to sequences~$u_k$ of punctured 
$J$-holomorphic curves for which $E(u_k)$ is uniformly bounded.
Such a bound exists for any sequence of curves with fixed genus, 
asymptotics and homology class.  Note also that the restriction of~$J$ to the
cylindrical end is also adapted to~$\lambda$ in the usual sense,
thus the upper level curves that appear in holomorphic buildings arising
from the compactness theorem are precisely the curves that are counted
in the definition of $H_*^\SFT(M,\lambda,\ff ; \rR)$.

The above observations yield the following generalization of
Proposition~\ref{prop:strongFillings}:

\begin{prop}
\label{prop:stableFillings}
Suppose $(X,\omega)$ is a stable symplectic
filling of $(M,\xi)$, and $\rR(X) \subset
\ker\omega \subset H_2(X;\RR)$ and $\rR \subset H_2(M;\RR)$ are linear
subspaces such that the natural map $H_2(M;\RR) \to H_2(X;\RR)$ takes
$\rR$ into~$\rR(X)$.  Then there exists a natural $\RR[[\hbar]]$-module
morphism
$$
\Phi_X : H_*^\SFT(M,\xi;\rR) \to \overline{R_{\rR(X)}}[[\hbar]],
$$
which acts on $R_\rR[[\hbar]]$ as the natural map to $R_{\rR(X)}[[\hbar]]$
induced by the inclusion $M \hookrightarrow X$.
In particular, defining a $2$-form on~$M$
by $\Omega = \omega|_M$, the $\Omega$-twisted SFT
of $(M,\xi)$ admits an $\RR[[\hbar]]$-module morphism
$$
\Phi_X : H_*^\SFT(M,\xi,\Omega) \to \overline{R_{\ker\omega}}[[\hbar]].
$$
\end{prop}

\begin{example}
\label{ex:classifyFillings}
The following shows that aside from defining filling obstructions,
SFT can also provide information as to the classification of symplectic 
fillings.
Consider for instance the tight contact structure~$\xi_0$ on $S^1 \times S^2$, 
which it aquires as the boundary of the Stein domain 
$S^1 \times B^3 \subset T^*S^1 \times \RR^2$.  
Presenting $(S^1 \times S^2,\xi_0)$ via a symmetric summed open book 
with disk-like pages (see Definition~\ref{defn:symmetric}), 
one can find a Reeb orbit that is 
uniquely spanned by two rigid holomorphic planes whose homology classes
differ by the generator $[S^2] := [\{*\} \times S^2] \in 
H_2(S^1 \times S^2;\RR)$. Hence, in the notation established at the
 beginning of this section, the fully twisted SFT
satisfies a relation of the form
$$
[ 1 - z^{[S^2]} ] = 0 \in H_*^\SFT(S^1 \times S^2,\xi_0 ; \{0\}).
$$
Then if $(X,\omega)$ is any weak filling of $(S^1 \times S^2,\xi_0)$,
Proposition~\ref{prop:stableFillings} gives a map from
$H_*^\SFT(S^1 \times S^2,\xi_0 ; \{0\})$ to the Novikov completion of
$\RR[H_2(X;\RR)]$ whose action on $\RR[H_2(M;\RR)][[\hbar]]$ is determined
by the inclusion $S^1 \times S^2 \hookrightarrow X$.  In light of the above
relation, this implies that the natural map $H_2(S^1 \times S^2;\RR) \to
H_2(X;\RR)$ takes $[S^2]$ to zero.  In fact, this is known to be true: 
it follows from the disk filling argument of Eliashberg 
\cite{Eliashberg:diskFilling}, which implies that every
weak filling of $S^1 \times S^2$ is diffeormorphic to a blow-up of
$S^1 \times B^3$.

Another example is provided by the standard $3$-torus $(T^3,\xi_0)$,
which is the boundary of the Stein domain $T^2 \times \DD \subset T^*T^2$
and can also be presented by a symmetric summed open book, but with
cylindrical pages.  One can then choose a $1$-dimensional subspace
$\rR \subset H_2(T^3;\RR)$ with generator~$d_0$ represented by 
a pre-Lagrangian torus, so that counting holomorphic cylinders yields
relations of the form
$$
[ (1 - z^{d_1}) \hbar ] = [ (1 - z^{d_2}) \hbar ] = 0 \in
H_*^\SFT(T^3,\xi_0 ; \rR)
$$
for both of the other canonical generators $d_1,d_2 \in H_2(T^3;\RR)$.
Applying Proposition~\ref{prop:stableFillings} again, one can use this to
show that
for any weak filling $(X,\omega)$ of $(T^3,\xi_0)$ such that
$\int_{d_0} \omega = 0$, and in particular for any strong filling,
the natural map $H_2(T^3,\RR) \to H_2(X;\RR)$
has its image in a space of dimension at most one.  This is also known to
be true: by a combination of arguments in \cite{Wendl:fillable}
and \cite{NiederkruegerWendl}, $(X,\omega)$ must in this case be a
symplectic blow-up of the standard Stein filling $T^2 \times \DD$.
\end{example}

\subsection{Algebraic torsion and its consequences}

As above, we write $\rR$ for some given linear
  subspace in $H_2(M;\RR)$, and use the notation $R_\rR =
\RR[H_2(M;\RR) / \rR]$ for the corresponding group ring.  Recall the
following definition from the introduction: 
\begin{defn}
For any integer $k \ge 0$, we say that $(M,\xi)$ has
algebraic torsion of order $k$ with coefficients in $R_\rR$
if $[\hbar^k] = 0$ in $H^\SFT_*(M,\xi ;\rR)$.  We single out the
following special cases:
\begin{itemize}
\item $(M,\xi)$ has (untwisted) algebraic $k$-torsion if
$[\hbar^k] = 0 \in H_*^\SFT(M,\xi)$.
\item For a closed $2$-form $\Omega$ on~$M$, 
$(M,\xi)$ has \emph{$\Omega$-twisted} algebraic $k$-torsion if
$[\hbar^k] = 0 \in H_*^\SFT(M,\xi,\Omega)$.
\item $(M,\xi)$ has \emph{fully twisted} algebraic $k$-torsion if
$[\hbar^k] = 0 \in H_*^\SFT(M,\xi ; \{0\})$.
\end{itemize}
\end{defn}
By default, when we speak of algebraic torsion without specifying
the coefficients, we will always mean the untwisted version.
Observe that due to the morphisms \eqref{eqn:coefficients}, fully
twisted torsion implies $\Omega$-twisted torsion for all closed $2$-forms
$\Omega$, and it is not hard to show that the converse is also true.
Likewise, $\Omega$-twisted torsion for \emph{any one} closed $2$-form $\Omega$ 
implies untwisted torsion, and $k$-torsion for any choice of 
coefficients implies $(k+1)$-torsion for the same coefficients since 
$\Dsft(Q)=\hbar^k$ implies $\Dsft(\hbar Q) = \hbar^{k+1}$.

\begin{remark}
\label{remark:preferred}
Since all power series in $\RR[[\hbar]]$ are naturally closed elements of the
SFT chain complex, one can define a seemingly more general notion than
algebraic torsion via the condition
$$
[f(\hbar)] = 0 \in H_*^\SFT(M,\xi)
$$
for any nonzero power series $f \in \RR[[\hbar]]$.  In fact, this is not
more general: all elements of the form $1 + \Order(\hbar)$ can be
inverted in $\RR[[\hbar]]$ via alternating series, thus $[f(\hbar)] = 0$
implies untwisted algebraic $k$-torsion where $k \ge 0$ is the largest
integer with $f(\hbar) = \hbar^k g(\hbar)$ for some
$g \in \RR[[\hbar]]$.  
The situation changes when one
considers the vanishing of nonzero elements of $R_\rR[[\hbar]]$ in
$H_*^\SFT(M,\xi ; \rR)$: as shown by Example~\ref{ex:classifyFillings}
above, this does not always imply nonfillability, but it can
yield topological restrictions on the symplectic fillings that exist.
\end{remark}

The special case $k=0$ is not a new concept; the following result
is stated for the untwisted theory but has obvious analogues for any
choice of coefficients~$R_\rR$.
\begin{prop}
\label{prop:algOT}
The following statements are equivalent.
\begin{enumerate}
\renewcommand{\labelenumi}{(\roman{enumi})}
\item $(M,\xi)$ has algebraic $0$-torsion.
\item $H^\SFT_*(M,\xi) = 0$.
\item $(M,\xi)$ is \emph{algebraically overtwisted} in the sense of
\cite{BourgeoisNiederkrueger:algebraically}, i.e.~its contact
homology is trivial.
\end{enumerate}
\end{prop}
\begin{proof}
The only claim not
immediate from the definitions is that~(i) 
implies~(ii), for which we use a variation on the main argument 
in \cite{BourgeoisNiederkrueger:algebraically}.
For $Q_1,Q_2\in \aA[[\hbar]]$, define
$$
[Q_1,Q_2]:= \Dsft(Q_1Q_2) - \Dsft(Q_1)Q_2 -(-1)^{|Q_1|}Q_1\Dsft(Q_2)
$$
to be the deviation of $\Dsft$ from being a derivation. Note that since
the first term $D_1$ in the expansion of $\Dsft$ is a derivation, we
always have $[Q_1,Q_2] = \Order(\hbar)$. One also easily checks that $\Dsft$ is
a derivation of this bracket, in the sense that
$$
\Dsft[Q_1,Q_2] = -[\Dsft Q_1,Q_2] -(-1)^{|Q_1|}[Q_1,\Dsft Q_2].
$$
These signs are correct because the bracket has odd degree. 

Now suppose $\Dsft(P)=1$, and define a map $B:\aA[[\hbar]]
\to \aA[[\hbar]]$ as an alternating sum of iterated brackets with $P$, i.e. as
$$
B(Q) := Q - [P,Q] + [P,[P,Q]] - \dots
$$
Clearly $[P,B(Q)]= Q-B(Q)$ and $\Dsft(B(Q))=B(\Dsft(Q))$, and so if 
$\Dsft(Q)=0$, then 
$$
\Dsft(P\cdot B(Q)) = [P,B(Q)] + \Dsft(P)\cdot B(Q) = Q-B(Q)+B(Q)=Q,
$$
proving that every closed element in $\aA[[\hbar]]$ is exact.
\end{proof}

With the algebraic formalism in place, the proofs of
Theorems~\ref{thm:obstructions} and~\ref{thm:stable} are now immediate.

\begin{proof}[Proofs of Theorems~\ref{thm:obstructions} and~\ref{thm:stable}]
Suppose $(X,\omega)$ is an exact symplectic cobordism with positive end
$(M^+,\xi^+)$ and negative end $(M^-,\xi^-)$.  Then if
$[\hbar^k] = 0 \in H_*^\SFT(M^+,\xi^+)$, the same must be true in
$H_*^\SFT(M^-,\xi^-)$ due to Proposition~\ref{prop:exactCobordisms}.

Likewise, if $(X,\omega)$ is a strong filling of $(M,\xi)$, then
Proposition~\ref{prop:strongFillings} gives an $\RR[[\hbar]]$-module
morphism from $H_*^\SFT(M,\xi)$ to $\overline{R_{\rR(X)}}[[\hbar]]$, where
$\overline{R_{\rR(X)}}$ is the Novikov completion of
$\RR[H_2(X;\RR) / \ker\omega]$.  Since no power of $\hbar$ vanishes
in $\overline{R_{\rR(X)}}[[\hbar]]$, the same must be true in
$H_*^\SFT(M,\xi)$, completing the proof of Theorem~\ref{thm:obstructions}.
Theorem~\ref{thm:stable} follows by exactly the same argument, using
Proposition~\ref{prop:stableFillings} and observing that
$H_*^\SFT(M,\xi,\Omega)$ depends only on $(M,\xi)$ and the cohomology
class of~$\Omega$.
\end{proof}

\section{Relation to planar torsion in dimension 3}
\label{sec:planar}
This section describes the relation of algebraic torsion to planar
torsion, and in particular provides the proof of
Theorem~\ref{thm:torsion}.

\subsection{Review of planar torsion}
We begin by reviewing briefly the notion of planar torsion, which is
defined in more detail in \cite{Wendl:openbook2}.  
A \emph{planar torsion domain} is a special type
of contact manifold with boundary which generalizes the thickened torus
$(T^2 \times [0,1],\xi_T)$ in the definition of Giroux torsion.  We can
define it in terms of open book decompositions as follows.

Recall first that if $\check{M}$ is a closed oriented (not necessarily
connected) $3$-manifold with an open
book decomposition $\check{\pi} : \check{M} \setminus \check{B} \to S^1$, 
then the open book can be
``blown up'' along part of its binding to produce a manifold with boundary: 
for any given binding component $\gamma \subset \check{B}$, this means replacing
$\gamma$ with its unit normal bundle.  The latter is then a $2$-torus~$T$
in the boundary of the blown up manifold~$M$, and it comes
with a canonical homology basis $\{\mu,\lambda\} \subset H_1(T)$, where
$\mu$ is the meridian around the boundary of a neighborhood of~$\gamma$
and $\lambda$ is a boundary component of a page.  Given any two
binding components $\gamma_1, \gamma_2 \subset \check{B}$, one can then produce
a new manifold via a so-called \emph{binding sum}, which consists of
the following two steps:
\begin{enumerate}
\item Blow up at $\gamma_1$ and $\gamma_2$ to produce boundary
tori $T_1$ and $T_2$ with canonical homology bases
$\{ \mu_1,\lambda_1\}$ and $\{\mu_2,\lambda_2\}$ respectively.
\item Attach $T_1$ to $T_2$ via an orientation reversing
diffeomorphism $T_1 \to T_2$ that maps
$\lambda_1$ to $\lambda_2$ and $\mu_1$ to $-\mu_2$.
\end{enumerate}
Combining both the blow-up and binding sum operations for a given
closed manifold with an open book $\check{\pi} : \check{M}
\setminus \check{B} \to S^1$, one obtains a compact manifold
$M$, possibly with boundary, carrying a fibration
$$
\pi : M \setminus (B \cup \iI) \to S^1,
$$
where~$B$ is an oriented (possibly empty) link consisting of all 
components of $\check{B}$ that have not been blown up, and~$\iI$ is 
a special (also possibly empty) collection of $2$-tori 
which are each the result of identifying two blown up
binding components in a binding sum.  The tori 
$T \subset \iI \cup \p M$ each carry canonical homology 
bases $\{ \mu,\lambda\} \subset
H_1(T)$, where for $T \in \iI$, $\mu$ is defined only up
to a sign.  These homology bases together with the fibration~$\pi$
determine a so-called 
\emph{blown up summed open book}~$\boldsymbol{\pi}$ on~$M$, with
\emph{binding}~$B$ and \emph{interface}~$\iI$.  Its \emph{pages} are
the connected components of the fibers $\pi^{-1}(\text{const})$.
We call a blown up summed open book \emph{irreducible} if the fibers
$\pi^{-1}(\text{const})$ are connected, which means it contains only a
single $S^1$-family of pages.  In general, every manifold~$M$ with a blown 
up summed open book~$\boldsymbol{\pi}$
can be written as a union of \emph{irreducible subdomains},
$$
M = M_1 \cup \ldots \cup M_n,
$$
where $M_i$ are manifolds with boundary that each carry 
irreducible blown up summed open books
$\boldsymbol{\pi}_i$, whose pages are pages of~$\boldsymbol{\pi}$, and
they are attached to each other along tori in the 
interface of~$\boldsymbol{\pi}$.

Just as an open book on $M$ determines a special class of contact forms, 
we define a \emph{Giroux form} on a manifold~$M$ with a blown up summed  
open book to be any contact form~$\lambda$ with the following properties:
\begin{enumerate}
\item The Reeb vector field~$X_\lambda$ is everywhere
positively transverse to the pages and positively tangent to the oriented
boundaries of their closures,
\item The characteristic foliation cut out by $\xi=\ker\lambda$ on each
boundary or interface torus $T \subset \iI \cup \p M$ has closed leaves
in the homology class of the meridian.
\end{enumerate}
Note that whenever $\lambda$ is a Giroux form, the binding consists of
periodic orbits of~$X_\lambda$, and each torus in $\iI \cup \p M$ is
foliated by periodic orbits.
A Giroux form can be defined for any blown up summed open book that contains
no closed pages, and it is then unique up to deformation.  We say that 
a contact structure~$\xi$ on~$M$ is
\emph{supported} by a given blown up summed open book if and only if
it can be written
as the kernel of a Giroux form.  The effect of a binding sum on supported
contact structures is then equivalent to a special case of the 
\emph{contact fiber sum} defined by Gromov \cite{Gromov:PDRs} and 
Geiges \cite{Geiges:constructions}.

\begin{defn}\label{defn:symmetric}
A blown up summed open book is called \emph{symmetric} if it has
no boundary and contains
exactly two irreducible subdomains, each with pages
of the same topological type, and each with empty binding and
(interior) interface.
\end{defn}
Symmetric examples are constructed in general by taking any two open books
with diffeomorphic pages, choosing an oriented diffeomorphism from the binding
of one to the binding of the other and
constructing the corresponding binding sum on their disjoint union.
Supported contact manifolds that arise in this way include the tight
$S^1 \times S^2$ (with disk-like pages) and the standard~$T^3$ (cylindrical
pages).

We call an irreducible blown up summed open book \emph{planar} if its
pages have genus~$0$, and a general blown up summed open book is then
\emph{partially planar} if it contains a planar irreducible subdomain
in its interior.

\begin{defn}
For any integer $k \ge 0$,
a \emph{planar torsion domain of order~$k$} (or simply
\emph{planar $k$-torsion domain}) is a connected contact
$3$-manifold $(M,\xi)$, possibly with boundary, with a supporting
blown up summed open book $\boldsymbol{\pi}$ such that:
\begin{enumerate}
\item
$M$ contains a planar irreducible subdomain $M^P \subset M$ in its
interior, whose pages have $k+1$ boundary components,
\item
$M \setminus M^P$ is not empty, and
\item
$\boldsymbol{\pi}$ is not symmetric.
\end{enumerate}
We then call the subdomains $M^P$ and $\overline{M\setminus M^P}$ the
\emph{planar piece} and the \emph{padding} respectively.

A contact $3$-manifold is said to have \emph{planar $k$-torsion} whenever
it admits a contact embedding of a planar $k$-torsion domain.
\end{defn}

\begin{defn}
\label{defn:planarSeparating}
Suppose $(M,\xi)$ is a contact $3$-manifold containing a planar
$k$-torsion domain $M_0 \subset M$ with planar piece $M_0^P$ for
some $k \ge 0$, and $\Omega$ is a closed $2$-form on~$M$.  If every
interface torus $T \subset M_0$ lying in $M_0^P$ satisfies
$\int_T \Omega = 0$, then we say that $(M,\xi)$ has
\emph{$\Omega$-separating} planar $k$-torsion.  We say that $(M,\xi)$ has
\emph{fully separating} planar $k$-torsion if this is true for every
closed $2$-form on~$M$, or equivalently, each of the relevant
interface tori separates~$M$.
\end{defn}

\begin{figure}
\includegraphics{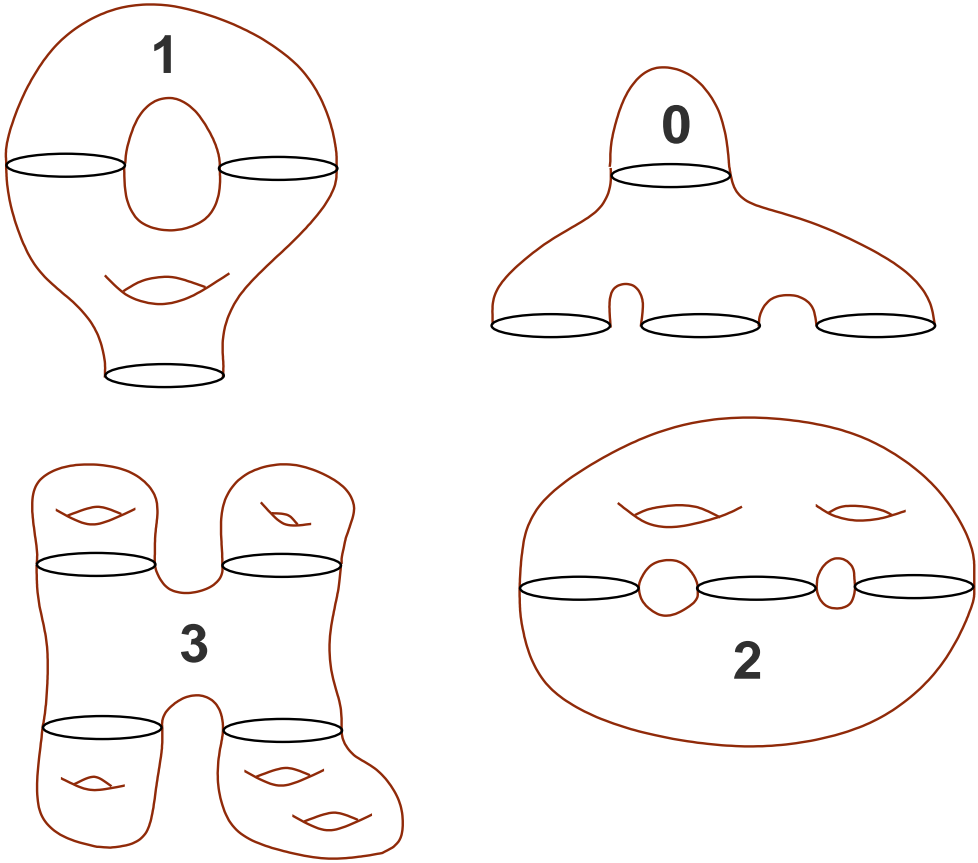}
\caption{\label{fig:torsionDomains}
Some examples of convex surfaces and dividing sets that determine 
$S^1$-invariant
planar torsion domains, of orders~$1$, $0$, $3$ and~$2$ respectively.  
The examples at the top right and bottom left are both fully
separating. 
The bottom right example defines a closed manifold contactomorphic to the
example $(V_4,\xi_3)$ from Theorem~\ref{thm:higherOrder}.  Note that in this
case, it's important that the two surfaces on either side of the dividing
set are not diffeomorphic (so that the summed open book is not
symmetric).}
\end{figure}

\begin{example}
\label{ex:planarTorsion}
The simplest examples of planar torsion domains have the form $S^1
\times \Sigma$, where $\Sigma$ is an orientable surface (possibly with
boundary), the contact structure is $S^1$-invariant and the resulting
dividing set $\Gamma \subset \Sigma$ contains the boundary.  This may
be viewed as a blown up summed open book whose pages are the connected
components of $\Sigma \setminus \Gamma$, so the binding is empty, and
the interface and boundary together are $S^1 \times \Gamma$.
Some special cases are shown in Figure~\ref{fig:torsionDomains}.
\end{example}

\begin{remark}
Another phenomenon that is allowed by the definition but not seen in
the cases $S^1 \times \Surface$ of Example~\ref{ex:planarTorsion} is for an
irreducible subdomain to have interface tori in its interior, due to
summing of a single connected open book to itself at different binding
components.  Examples of this are shown in 
Figure~\ref{fig:anotherTorsionDomain}, which also illustrates the fact
that the choice of planar piece (and consequently the order of planar
torsion) is not always unique, even for a fixed planar torsion domain.
\end{remark}

\begin{figure}
\includegraphics{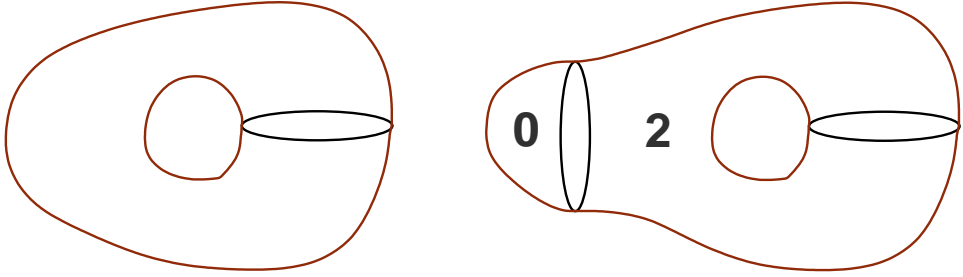}
\caption{\label{fig:anotherTorsionDomain}
Schematic representations of two summed open books that include
``self summing'', i.e.~interface tori in the interior of an irredubicle
subdomain.  Assuming trivial monodromy, the example at the left is obtained
from the tight $S^1 \times S^2$ with its obvious cylindrical open book by
summing one binding component to the other: the result is a Stein fillable
contact structure on the torus bundle over~$S^1$ with monodromy~$-1$.
At the right, the additional subdomain with disk-like pages turns it into a
planar torsion domain: the $3$-manifold is the same, but the contact
structure is changed by a half Lutz twist and is thus overtwisted.  
Note that in this example either
irreducible subdomain can be taken as the planar piece,
so it is both a $0$-torsion domain and a $2$-torsion domain.
}
\end{figure}

It is shown in \cite{Wendl:openbook2} that a contact manifold has
planar $0$-torsion if and only if it is overtwisted, and every contact
manifold with Giroux torsion also has planar $1$-torsion.  The latter
is the reason why Theorem~\ref{thm:torsion} implies Theorem~\ref{thm:Giroux}.

\subsection{Proof of Theorem~\ref{thm:torsion}}
With these definitions in place, Theorem~\ref{thm:torsion} follows easily from
an existence and uniqueness result proved in \cite{Wendl:openbook2} for 
$J$-holomorphic curves in blown up summed open books.  Namely, suppose
$(M,\xi)$ is a closed contact $3$-manifold containing a compact and
connected $3$-dimensional submanifold
$M_0$, possibly with boundary, on which $\xi$ is supported by a blown up
summed open book $\boldsymbol{\pi}$ with binding~$B$, interface~$\iI$ and 
induced fibration
$\pi : M_0 \setminus (B \cup \iI) \to S^1$.  Assume there are 
$N \ge 2$ irreducible subdomains
$$
M_0 = M_1 \cup \ldots \cup M_N,
$$
of which $M_1$ lies fully in the interior of~$M_0$,
and denote the corresponding restrictions of~$\pi$ by
$$
\pi_i : M_i \setminus (B_i \cup \iI_i) \to S^1
$$
for $i=1,\ldots,N$, with $B_i := B \cap M_i$ and $\iI_i :=
\iI \cap \interior{M_i}$.  Note that while~$\pi$ itself is not 
necessarily well defined at $\p M_i$, $\pi_i$ always has a continuous extension
to~$\p M_i$.
Assume the pages in~$M_i$ have genus $g_i \ge 0$,
where $g_1 = 0$.  In particular, $M_0$ is a planar torsion
  domain with planar piece $M_1$.

\begin{prop}[\cite{Wendl:openbook2}]
\label{prop:openbook2}
For any number $\tau_0 > 0$, $(M,\xi)$ admits a Morse-Bott contact
form $\lambda$ and compatible Fredholm regular almost complex structure $J$
with the following properties.
\begin{enumerate}
\item On $M_0$, $\lambda$ is a Giroux form for~$\boldsymbol{\pi}$.
\item The Reeb orbits in $B$ are nondegenerate and elliptic, and the
components of $\iI \cup \p M_0$ are all Morse-Bott submanifolds.
\item All Reeb orbits in $B_1 \cup \iI_1 \cup \p M_1$ have minimal period at
most~$\tau_0$, while every other closed orbit of the Reeb vector
field~$X_\lambda$ in~$M$ has minimal period at least~$1$.
\item For each irreducible subdomain $M_i$ with $g_i = 0$, the fibration
$\pi_i : M_i \setminus (B_i \cup \iI_i) \to S^1$ admits a $C^\infty$-small 
perturbation
$\hat{\pi}_i : M_i \setminus (B_i \cup \iI_i) \to S^1$ such that the 
interior of each fiber $\hat{\pi}_i^{-1}(\tau)$ for $\tau \in S^1$ lifts 
uniquely to an $\RR$-invariant family of properly embedded surfaces
$$
S^{(i)}_{\sigma,\tau} \subset \RR \times M_i,
\qquad (\sigma,\tau) \in \RR\times S^1,
$$
which are the images of embedded finite energy $J$-holomorphic curves
$$
u^{(i)}_{\sigma,\tau} = (a^{(i)}_\tau + \sigma, F^{(i)}_\tau)
: \dot{\RS}_i \to \RR \times M_i,
$$
all of them Fredholm regular with index~$2$, and with only positive ends.
\item Suppose $u : \dot{\RS} \to \RR \times M$ is a finite energy
punctured $J$-holomorphic curve which is not a cover of a trivial
cylinder, and such that all its positive asymptotic orbits  are
simply covered and contained in $B_1 \cup \iI_1 \cup \p M_1$, with at most one
positive end approaching each connected component of $B_1 \cup \p M_1$ 
and at most two approaching each connected component of~$\iI_1$.
Then $u$ has genus zero and parametrizes one of the
surfaces $S^{(i)}_{\sigma,\tau}$ described above.
\end{enumerate}
\end{prop}

Recall that a $J$-holomorphic curve is called \emph{Fredholm regular}
if it corresponds to a transversal intersection of the appropriate
section of a Banach space bundle with the zero-section, see for example
\cite{Wendl:automatic}.  We also say that~$J$ is Fredholm regular
if every somewhere injective $J$-holomorphic curve is Fredholm regular;
this is a generic condition due to \cite{Dragnev}.
If $u$ is a rigid curve that is Fredholm regular, this implies in
particular that~$u$ can be perturbed uniquely to a solution of any sufficiently
small perturbation of the nonlinear Cauchy-Riemann equation.

\begin{proof}[Proof of Theorem~\ref{thm:torsion}]
The following is an adaptation of the argument used in
\cite{Wendl:openbook2} to show that planar torsion kills the ECH
contact invariant, and it can similarly be used to compute an upper
bound on the integer $f_{\text{simp}}^T(M,\lambda,J)$ 
defined via ECH in the appendix.
Given a closed $2$-form~$\Omega$ on~$M$,
let $k_0 \leq k$ be the smallest order of $\Omega$-separating
planar torsion that $(M,\xi)$ admits.
We will prove that $(M,\xi)$ then has $\Omega$-twisted 
algebraic $k_0$-torsion, which as previously observed, implies algebraic 
$k$-torsion.  The statement for untwisted algebraic torsion is then
the special case where $\Omega=0$. Throughout the proof, for any $d \in
H_2(M;\RR)$, we denote by 
$$
\bar{d} \in H_2(M ; \RR) / \ker\Omega
$$
the corresponding equivalence class.

Suppose $M_0 \subset (M,\xi)$ is a planar $k_0$-torsion domain with
planar piece $M_0^P \subset M_0$,  
such that $[T] \subset \ker\Omega \subset H_2(M;\RR)$ for every
interface torus~$T$ lying in~$M_0^P$.
Denote by
$$
\pi^P : M_0^P \setminus (B^P \cup \iI^P) \to S^1
$$
the corresponding fibration in the planar piece.  Write the connected
components of the binding, interface and boundary respectively as
\begin{equation*}
\begin{split}
B^P &= \gamma_1 \cup \ldots \cup \gamma_m,\\
\p M_0^P &= T_1 \cup \ldots \cup T_n,\\
\iI^P &= T_{n+1} \cup \ldots \cup T_{n+r},
\end{split}
\end{equation*}
where by definition we have
$$
m + n + 2r = k_0 + 1 \quad\text{ and }\quad n \ge 1.
$$
Now given the special Morse-Bott contact form $\lambda_0$ and compatible
almost complex structure~$J_0$ provided by Proposition~\ref{prop:openbook2},
we consider the moduli space
$$
\mM(J_0) := \mM_0(\gamma_1,\ldots,\gamma_m,T_1,\ldots,T_n,
T_{n+1},T_{n+1},\ldots,T_{n+r},T_{n+r} ; J_0)
$$
of unparametrized $J_0$-holomorphic curves $u : \dot{\RS} \to \RR\times M$
such that
\begin{enumerate}
\item
$\dot{\RS}$ has genus~$0$, no negative punctures and $m+n+2r$ positive 
punctures
$$
z_1,\ldots,z_m,\zeta_1,\ldots,\zeta_n,w_1^+,w_1^-,\ldots,w_r^+,w_r^-
$$
\item
For the punctures listed above, $u$ approaches the simply covered 
orbit~$\gamma_i$ at~$z_i$, any simply covered orbit in $T_i$
at $\zeta_i$ and any simply covered orbit in~$T_{n+i}$ at both
$w_i^+$ and $w_i^-$.
\end{enumerate}
By Prop.~\ref{prop:openbook2}, $\mM$ is a connected $2$-dimensional
manifold consisting of an $\RR$-invariant family of embedded Fredholm 
regular curves that project to the pages in~$M_0^P$.  Note here we are
using the fact that the blown up summed open book on $M_0$ is not
symmetric, so in particular the padding $M_0 \setminus M_0^P$ cannot
contain additional genus~$0$ curves with the asymptotic behavior 
that defines~$\mM(J_0)$.  It also cannot contain any genus~$0$ curves 
asymptotic to a proper subset of the same orbits, as this would mean
the existence of an $\Omega$-separating
planar torsion domain with order less than~$k_0$.

We next perturb the Morse-Bott data $(\lambda_0,J_0)$ 
to generic nondegenerate data $(\lambda,J)$ by the 
scheme described in \cite{Bourgeois:thesis}, extend~$J$ to a
suitable framing~$\ff$ and assume that
$H_*^\SFT(M,\lambda,\ff,\Omega)$
is well defined (see Remark~\ref{remark:abstract}
below).  Recall that the perturbation to nondegenerate data
is achieved by choosing
a Morse function on each of the relevant Morse-Bott families of orbits
and using it to alter the contact form in small neighborhoods of these
families.  In our case, each Morse-bott family is parametrized by a circle,
so we may assume without loss of generality that our Morse function on~$S^1$
has exactly two critical points, which correspond to the two orbits in
the family that survive as nondegenerate orbits after the perturbation.
Moreover, $J$-holomorphic curves are obtained as
perturbations of $J_0$-holomorphic ``cascades'', i.e. multi-level buildings 
composed of a mixture of holomorphic curves with gradient flow lines along 
the Morse-Bott manifolds.  We may therefore assume after the
perturbation that each of the tori $T_i$ for $i=1,\ldots,n+r$
contains two nondegenerate simple Reeb orbits $\gamma_i^e$ and
$\gamma_i^h$, elliptic and hyperbolic respectively.  These orbits
come with preferred framings determined by the tangent spaces to~$T_i$,
and in these framings their Conley-Zehnder indices are
$$
\muCZ(\gamma_i^e) = 1 \quad\text{ and }\quad
\muCZ(\gamma_i^h) = 0.
$$
There are also two embedded $J$-holomorphic index~$1$ cylinders
(corresponding to gradient flow lines along the Morse-Bott family)
$$
v_i^\pm : \RR \times S^1 \to \RR \times M
$$
whose projections to~$M$ are disjoint and fill the two regions in~$T_i$
separated by $\gamma_i^e$ and $\gamma_i^h$, so the homology classes
they represent are related to each other by
$$
[v_i^+] - [v_i^-] = [T_i] \in H_2(M;\RR),
$$
and for a suitable choice of coherent orientation, these two together
contribute terms of the form
$$
(z^{\overline{[T_i]}} - 1) q_{\gamma_i^h} \frac{\p}{\p q_{\gamma_i^e}}
$$
to the operator $\Dsft$.
The curves in $\mM(J_0)$ likewise give rise to a unique
$J$-holomorphic punctured sphere in the space
$$
\mM(J) := \mM_0(\gamma_1,\ldots,\gamma_m,\gamma_1^h,\gamma_2^e,\ldots,\gamma^e_n,
\gamma^e_{n+1},\gamma^e_{n+1},\ldots,\gamma^e_{n+r},\gamma^e_{n+r} ; J)
$$
with puncture $\zeta_1$ asymptotic to $\gamma^h_1$ and all other punctures
asymptotic to elliptic orbits.  This curve is embedded and has
index~$1$, thus if $d \in H_2(M;\RR)$ denotes the homology class defined by
the pages in $M_0^P$ with attached capping surfaces, then this curve
produces a term
$$
z^{\bar{d}} \hbar^{m+n+2r-1} \frac{\p}{\p q_{\gamma_1^h}}
\prod_{i=1}^m \frac{\p}{\p q_{\gamma_i}}
\prod_{i=2}^n \frac{\p}{\p q_{\gamma_i^e}}
\prod_{i=1}^r {\frac 1 2} \frac{\p}{\p q_{\gamma_{n+i}^e}}\frac{\p}{\p q_{\gamma_{n+i}^e}}
$$
in $\Dsft$.  We thus define the monomial
$$
F = q_{\gamma_1}\ldots q_{\gamma_m} q_{\gamma_1^h} 
q_{\gamma_2^e} \ldots q_{\gamma_n^e}
q_{\gamma_{n+1}^e} q_{\gamma_{n+1}^e} \ldots
q_{\gamma_{n+r}^e} q_{\gamma_{n+r}^e}
$$
and compute,
$$
\Dsft F = z^{\bar{d}} \hbar^{k_0} + 
\sum_{i=2}^{n+r} ( z^{\overline{[T_i]}} - 1) q_{\gamma_i^h} 
\frac{\p F}{\p q_{\gamma_i^e}}.
$$
Every term in the summation now vanishes since
$[T_i] \subset \ker\Omega$, implying that
$\hbar^{k_0}$ is exact.
\end{proof}

\begin{remark}
\label{remark:abstract}
To make the above computation fully rigorous, one must show that the relevant
count of curves doesn't change under a suitable abstract perturbation,
e.g.~as provided by \cite{Hofer:polyfolds}.  The curves that were counted
in the above argument are Fredholm regular and will thus survive any
such perturbation, but we also need to check that no additional curves
appear.  If any such curves exist, then in the unperturbed limit they must
give rise to nontrivial holomorphic \emph{cascades} in the natural 
compactification of~$\mM(J_0)$, 
see~\cite{SFTcompactness}.  It suffices therefore to observe that in
the above setup, all possible cascades are accounted for by the
$J_0$-holomorphic pages in~$M_0^P$,
due to the uniqueness statement in Prop.~\ref{prop:openbook2}.
\end{remark}

\section{$S^1$-invariant examples in dimension 3}
\label{sec:higherOrder}

In this section we consider the special examples
$(S^1 \times \Surface, \xi_\Gamma)$
described in the introduction, and prove in particular
Theorems~\ref{thm:noGiroux} and~\ref{thm:higherOrder}.
Note that the examples $(V_g,\xi_k)$ of Theorem~\ref{thm:higherOrder} can be
constructed via a summed open book as follows.  Fix $g\geq k \ge 1$, and let
$(M_-,\xi_-)$ denote the closed contact $3$-manifold supported by
a planar open book $\pi_- : M_- \setminus B_- \to S^1$ with $k$ binding 
components and trivial monodromy.
Similarly, let $(M_+,\xi_+)$ be the contact manifold supported by an
open book $\pi_+ : M_+ \setminus B_+ \to S^1$
with pages of genus $g-k+1 > 0$, $k$ binding components and
trivial monodromy.  Choosing any one-to-one correspondence between the
connected components of $B_+$ and $B_-$, we produce a new closed contact
manifold $(M,\xi)$ by taking the binding sum of $(M_+,\xi_+) \sqcup
(M_-,\xi_-)$ along corresponding binding components as described
in \S\ref{sec:planar}; this produces a closed planar $(k-1)$-torsion domain
which is contactomorphic to $(V_g,\xi_k)$.

To complete the proof of Theorem~\ref{thm:higherOrder}, we will
have to show that certain types of holomorphic curves in
$\RR\times V_g$ do \emph{not} exist (at least algebraically), which would
need to exist if $\hbar^{k-2}$ were exact 
(see Lemma~\ref{lemma:algHyperTorsionFree} below).  To do this, 
we will construct a precise model for contact manifolds of the form
$(S^1 \times \Surface, \xi_\Gamma)$, in which all the relevant holomorphic curves
can be classified.  The proof of Theorem~\ref{thm:noGiroux} will also
follow immediately from this classification.

\subsection{Holomorphic curves in $(S^1 \times \Surface,\xi_\Gamma)$}

The basic idea of our model for $(S^1 \times \Surface,\xi_\Gamma)$ will be
to choose data so that the singular foliation of~$\Surface$
defined by the gradient flow lines of a suitable Morse function gives
rise to a foliation of the symplectization by 
holomorphic cylinders, which can be counted by Morse homology.
We will then be able to exclude
all the other relevant curves by a combination
of intersection arguments and index estimates.

For the constructions carried out below, the following lemma turns out
to be convenient.
\begin{lemma}\label{lem:conf}
Suppose $\Surface$ is a compact connected oriented surface with nonempty boundary,
and $\tilde{h} : \Surface \to \RR$ is a smooth Morse function with all critical
points in the interior and none of index~$2$, and with 
$\p \Surface = \tilde{h}^{-1}(1)$.
Then there exists a conformal structure $j$ on $\Surface$, compatible with
the orientation, and a smooth, strictly increasing function 
$\varphi : \RR \to \RR$ 
such that $h := \varphi \circ \tilde{h}: \Surface \to \RR$ satisfies
$$
- d( dh \circ j) >0,
$$
and each boundary component has a collar neighborhood
biholomorphically identified with $(-\delta,0] \times S^1$ for some
small $\delta > 0$, so that in these holomorphic coordinates $(s,t) \in
(-\delta,0] \times S^1$ we have   
$$
h(s,t) = e^s.
$$
\end{lemma}
\begin{proof}
To construct $j$ with the required properties, we start by choosing
oriented coordinates $(s,t) \in (-2\delta,0] \x S^1$ on a collar neigbhorhood of
each boundary component such that $\tilde h(s,t)=e^s$ in these coordinates.
In this collar neighborhood, we simply define $j$ by requiring
$j(\p_s)=\p_t$ and $j(\p_t) = -\p_s$. Note that
$$
-d (d\tilde h \circ j) = e^s\, ds\wedge dt >0
$$
on these collars.

Next we choose oriented Morse coordinates near the
critical points, such that locally 
$$
\tilde h(x,y)= x^2 \pm y^2 + \tilde h(0).
$$
In such coordinates, we can define $j$ such that $j(\p_x)=\lambda
\p_y$ and $j(\p_y)=-\frac 1 \lambda \p_x$ for some $\lambda > 0$. 
A computation then yields
$$
-d (d\tilde h \circ j) = \left(\frac 2 \lambda \pm 2\lambda\right)\, 
dx \wedge dy,
$$
which is positive whenever $0<\lambda<1$. 

Now extend $j$ arbitrarily to all of $\Surface$ and consider the function
$h = \varphi \circ \tilde{h}$, where $\varphi : \RR \to \RR$ is a smooth
function with $\varphi' > 0$ and $\varphi'' \ge 0$.  Observe that the $2$-form
$$
\mu := - d\tilde{h} \wedge (d\tilde{h} \circ j)
$$
is everywhere nonnegative, and vanishes precisely at the critical points
of~$\tilde{h}$.  We then compute,
\begin{equation}
\label{eqn:convexity}
-d(dh \circ j) = - (\varphi' \circ \tilde{h}) \, d(d\tilde{h} \circ j) +
(\varphi'' \circ \tilde{h}) \mu.
\end{equation}
This is already positive whenever $-d(d\tilde{h} \circ j)$ is positive,
which is true on a neighborhood of the critical points and the boundary.
Outside of this neighborhood, we have $\mu > 0$ and can thus 
arrange $-d(dh \circ j) > 0$ by choosing $\varphi$ so that
$$
\frac{\varphi''}{\varphi'} \ge K
$$
for a sufficiently large constant $K > 0$.  Since $-d(d\tilde{h} \circ j) > 0$
on the collar neighborhoods $(-2\delta,0] \times S^1$ of $\p \Surface$, we are
free to set $\varphi'' = 0$ in $[-\delta,0] \times S^1$.  Now
since $-d(dh \circ j) > 0$ everywhere, \eqref{eqn:convexity} implies that
this property will survive a further postcomposition with an 
increasing affine function, hence through such a composition we can arrange
without loss of generality that $\varphi(s)=s$ on the collar neighborhoods
$[-\delta,0] \times S^1$.
\end{proof}

Let $\Surface_-$ and $\Surface_+$ denote compact oriented and possibly disconnected
surfaces, such that each connected component has non-empty boundary
and the total number of boundary components of $\Surface_-$ and $\Surface_+$ agrees.
On each of the surfaces $\Surface_\pm$, we choose a function $h_\pm$ and 
conformal structure $j_\pm$ as provided by the lemma and define a $1$-form by
$$
\beta_\pm = - d h_\pm \circ j_\pm.
$$
This induces a symplectic form $\sigma_\pm$ and Riemannian metric $g_\pm$ 
on $\Surface_\pm$, defined by
$$
\sigma_\pm = d\beta_\pm,
\qquad
g_\pm = \sigma_\pm(\cdot, j_\pm \cdot).
$$
Since $dh_\pm=e^s\,ds$ in holomorphic coordinates $(s,t) \in
(-\delta,0] \x S^1$ near each component of the boundary, we find
$$
\sigma_\pm=e^s\, ds \wedge dt, \quad \nabla h_\pm=\p_s.
$$
Denote the union of all these collar neighborhoods of $\p \Surface_\pm$
by 
$$
\uU_\pm \subset \Surface_\pm.
$$
The gradient $\nabla h_\pm$ 
is a Liouville vector field pointing orthogonally outward at $\p \Surface_\pm$.
\begin{remark}
\label{remark:MorseSmale}
Since the subharmonicity condition on the pair $(h_\pm,j_\pm)$ is
open, there is some freedom in the construction. In particular, by 
perturbing the conformal structure if necessary we can achieve that
the flow of $\nabla h_\pm$ is Morse-Smale. 
\end{remark}

We now glue $\Surface_+$ and $\Surface_-$ together along an orientation preserving
diffeomorphism $\p \Surface_+ \to \p \Surface_-$ to create a closed oriented
surface
$$
\Surface = \Surface_+ \cup (-\Surface_-),
$$
divided into two halves by a special set of circles
$\Gamma := \p \Surface_+ \subset \Surface$.  We will always assume~$\Surface$ is connected,
and as the above notation suggests we assign it the same
orientation as~$\Surface_+$, which is opposite the given orientation on~$\Surface_-$.
On each connected component of $\uU_+$ and $\uU_-$,
one can define new coordinates
\begin{equation*}
\begin{split}
S^1 \times [0,\delta) \ni (\theta,\rho) := (t,-s) & \text{ for $(s,t)
  \in \uU_+$},\\ 
S^1 \times (-\delta,0] \ni (\theta,\rho) := (t,s) & \text{ for $(s,t)
  \in \uU_-$}, 
\end{split}
\end{equation*}
and then define the gluing map and the smooth structure on~$\Surface$ so that 
each component of $\uU := \uU_+ \cup \uU_- \subset \Surface$
inherits smooth positively oriented
coordinates $(\theta,\rho) \in S^1 \times (-\delta,\delta)$.

Choose a function $g_0:[-\delta,\delta] \to \RR$ with
$g_0(\rho)= \pm 1$ for $\rho$ near $\pm \delta$, $g_0(0)=0$,
$g_0' \geq 0$ and $g_0' >0$ near $\rho=0$ and a function
$\gamma:[-\delta,\delta] \to \RR$ with $\gamma(\rho) = \mp 
e^{\mp \rho}$ for $\rho$ near $\pm \delta$, $ \gamma'>0$ wherever
$g_0'=0$, $\gamma(\rho) > 0$ for $\rho < 0$ and
$\gamma(\rho) < 0$ for $\rho > 0$.
For $\epsilon \in(0,1)$, we then set  
$$
g_\epsilon(\rho) = g_0(\rho) + \epsilon^2 \gamma(\rho),
$$
which  satisfies
\begin{itemize}
\item $g_\epsilon' > 0$ for sufficiently small $\epsilon>0$,
\item $g_\epsilon(\rho) = \pm (1 - \epsilon^2 e^{\mp \rho})$ for 
$\rho$ near~$\pm\delta$,
\item $g_\epsilon(0) = 0$.
\end{itemize}
Now define a smooth family of functions
$h_\epsilon : \Surface \to \RR$ by
$$
h_\epsilon = \begin{cases}
1 - \epsilon^2 h_+ & \text{ on $\Surface_+ \setminus \uU_+$,} \\
g_\epsilon(\rho) & \text{ for $(\theta,\rho) \in \uU$,} \\
-1 + \epsilon^2 h_- & \text{ on $\Surface_- \setminus \uU_-$.} \\
\end{cases}
$$

For each fixed $\epsilon > 0$, $h_\epsilon$ is a Morse
function with all its critical points in $\Surface\setminus \uU$, and they 
are precisely the critical points of~$h_\pm$.  

Next choose a function
$f_0: [-\delta,\delta] \to \RR$ such that $f_0(\rho) = 0$ for $\rho$
near $\pm \delta$, $f_0 \geq 0$ everywhere and $\rho \cdot f_0'(\rho)
\leq 0$ for $\rho \neq 0$ and $f_0''(0)<0$, and a function
$\psi:[-\delta,\delta] \to \RR$ with $\psi(\rho) = e^{\pm \rho}$ for
$\rho$ near $\mp \delta$, $\psi \geq 0$ everywhere and $\rho \cdot
\psi'(\rho) <0$ for $\rho \neq 0$.  
Then we define 
$$
f_\epsilon(\rho) = f_0(\rho) + \epsilon \psi(\rho).
$$

With these choices in place, we denote the coordinate in $S^1$ by $\phi$ and
define a smooth family of $1$-forms 
$\lambda_\epsilon$ on $S^1 \times \Surface$ by
\begin{equation}
\label{eqn:lambdaEpsilon}
\lambda_\epsilon = 
\begin{cases}
\epsilon \beta_+ + h_\epsilon\ d\phi & \text{ on $S^1 \times (\Surface_+ \setminus \uU_+)$,}\\
f_\epsilon(\rho)\ d\theta + g_\epsilon(\rho)\ d\phi & \text{ on $S^1 \times \uU$,} \\
\epsilon \beta_- + h_\epsilon\ d\phi & \text{ on $S^1 \times (\Surface_- \setminus \uU_-)$.}
\end{cases}
\end{equation}
Observe that $S^1 \times \Surface$ admits a natural summed open book with empty
binding, interface $\iI = S^1 \times \Gamma$, fibration
$$
\pi : S^1 \times (\Surface \setminus \Gamma) \to S^1 : (\phi,z) \mapsto
\begin{cases}
\phi & \text{ if $z \in \Surface_+$},\\
-\phi & \text{ if $z \in \Surface_-$},
\end{cases}
$$
and the meridians on $S^1 \times \Gamma$ generated by the circles
$S^1 \times \{\text{const}\}$.

\begin{prop}
\label{prop:contactForm}
There exists $\epsilon_0 > 0$ with the following properties.
\begin{enumerate}[(i)]
\item For any $\epsilon \in (0,\epsilon_0]$, $\lambda_\epsilon$ is a
  positive contact form on $S^1 \times \Surface$ and is a Giroux form for the
  summed open book described above. Moreover, for all these contact forms each
component of the interface $S^1 \times \Gamma$ is a Morse-Bott submanifold
of Reeb orbits pointing in the $\p_\theta$-direction.
\item  For any $\epsilon \in (0,\epsilon_0]$ and for each
$\phi \in S^1$, the leaves of the characteristic foliation on 
$\{\phi\} \times \Surface$ are precisely the gradient flow lines
of~$h_\epsilon$.
\item The $2$-form $\omega = d(e^s \lambda_s)$ is 
symplectic on $(0,\epsilon_0] \times S^1 \times \Surface$, where~$s$
denotes the coordinate on the first factor.
\end{enumerate}

\end{prop}
\begin{proof}
To prove (i), note that the natural co-orientation induced by
the summed open book on its pages is compatible with the
orientations defined on $\Surface_\pm$ by~$j_\pm$, for which $\sigma_\pm$ are positive
volume forms.  To prove the contact condition on $S^1 \times (\Surface_\pm
\setminus \uU_\pm)$, observe that $\lambda_\epsilon \to \pm d\phi$ on
this region as $\epsilon \to 0$, so the contact planes are almost
tangent to the pages.  Thus it suffices to observe that
$d\lambda_\epsilon$ is positive on $\Surface_\pm \setminus \uU_\pm$, 
which is clear since $d\lambda_\epsilon = \epsilon \sigma_\pm$  
when restricted to the pages.

On $S^1 \times \uU$, a routine computation shows that the contact condition 
follows from $f_\epsilon g_\epsilon' - f_\epsilon' g_\epsilon >
0$.  But this is easily computed to equal
$$
f_\epsilon g_\epsilon' - f_\epsilon' g_\epsilon = 
f_0 g_0'-f_0'g_0 + \epsilon (\psi g_0' - \psi' g_0) + \Order(\epsilon^2). 
$$
Our conditions on the various functions ensure that all four summands
are nonnegative, with the first one strictly positive for $\rho$ near 0 and the
last one strictly positive for $\rho$ away from zero. So for $\epsilon_0>0$
suffficiently small, the contact condition holds for all $\epsilon \in
(0,\epsilon_0]$ on $S^1 \times \uU$ as well.
Here it is also easy to compute the Reeb
vector field $X_{\lambda_\epsilon}$: writing 
$D_\epsilon = f_\epsilon g_\epsilon' - f_\epsilon' g_\epsilon$, we have
\begin{equation}
\label{eqn:Reeb1}
X_{\lambda_\epsilon}(\phi,\rho,\theta) = \frac{1}{D_\epsilon(\rho)} \left[
g_\epsilon'(\rho)\ \frac{\p}{\p \theta} - f_\epsilon'(\rho)\ \frac{\p}{\p \phi} \right].
\end{equation}
Our assumptions on $f_\epsilon'(\rho)$ then imply that $X_{\lambda_\epsilon}$ always
has a component in the negative $\p_\phi$-direction for $\rho \in (-\delta,0)$,
and in the positive $\p_\phi$-direction for $\rho \in (0,\delta)$, while at
$\rho=0$ it points in the $\p_\theta$-direction.  Moreover the condition
$g_\epsilon(0)=0$ implies that the contact planes at $\rho=0$ are tangent to the
circles $S^1 \times \{\text{const}\}$, thus
$\lambda_\epsilon$ is a Giroux form.  The Morse-Bott condition at
$S^1 \times \Gamma$ follows from $f_\epsilon''(0) < 0$, which
  for small $\epsilon>0$ follows from $f_0''(0)<0$. This concludes
the proof of (i).

Next we verify that the characteristic foliation on $\{\phi\} \times \Surface$
matches the gradient flow of~$h_\epsilon$.  This is obvious in
$\uU$, where both characteristic leaves and gradient flow lines are simply
straight lines in the $\p_\rho$-direction.  On $\Surface_\pm \setminus \uU_\pm$,
a vector $v \in T \Surface_\pm$ is tangent to the characteristic foliation if and
only if $\beta_\pm(v) = 0$, implying $dh_\pm (j_\pm v) = 0$ and thus
$v$ is orthogonal to the level sets of $h_\pm$, which makes it proportional
to $\nabla h_\pm$ as claimed, and establishes (ii).

Finally, consider the two-form $\omega = d(e^s \lambda_s)$. On
$\RR \times S^1 \times \uU$, we have $\lambda_s = f_s \,
d\theta + g_s \, d\phi$ and so 
$$
\omega = e^s(ds \wedge \lambda_s + df_s \wedge d\theta + dg_s
\wedge d\phi),
$$ 
with
\begin{align*}
df_s &= f_s' \, d\rho + \psi \, ds\\
dg_s &= g_s' \, d\rho + 2 s \gamma \, ds.
\end{align*}
One then computes 
$$
\omega \wedge \omega = e^s(f_s g_s' - f_s'g_s + \psi g_s' - 2\gamma s
f_s') \,  ds \wedge d\theta \wedge d\rho \wedge d\phi
$$
here, and observe that all four terms are nonnegative, with the first
one strictly positive for small $s>0$, so $\omega$ is symplectic here.

On $\RR \times S^1 \times (\Surface_+ \setminus \uU)$, we have
$\lambda_s = s \beta_+ + (1- s^2h_+) \, d\phi$, and so another
computation shows
$$
\omega \wedge \omega = e^{2s}(s\, \sigma_+\wedge ds \wedge d\phi +
\Order(s^2))
$$
here, which is also a positive volume form for small enough $s>0$. A similar
computation on $\RR \times S^1 \times (\Surface_- \setminus \uU)$
finishes the proof of part (iii).

\end{proof}

From now on, denote the contact structure on $S^1 \times \Surface$ for
$\epsilon \in (0,\epsilon_0]$ by
$$
\xi_\epsilon = \ker\lambda_\epsilon.
$$
Due to Gray's stability theorem, $\xi_\epsilon$ is independent of 
$\epsilon$ up to isotopy, and it is isomorphic to $\xi_\Gamma$.

\begin{remark}
\label{remark:Euler}
From the discussion above it is clear that for every $\phi \in S^1$,
$\{\phi\} \times \Surface$ is a convex surface for
$\xi_\epsilon$ with dividing set $\Gamma$, positive part $\Surface_+$ and
negative part $\Surface_-$. In particular, the Euler class
$e(\xi_\epsilon) \in H^2(S^1 \times \Surface)$ satisfies
$\langle e(\xi_\epsilon),[\{*\} \times \Surface] \rangle = \chi(\Surface_+)-\chi(\Surface_-)$.
It follows from the $S^1$-invariance of $\xi_\epsilon$ that the Euler
class vanishes on all cycles of the form $S^1 \x \gamma$ for
closed curves $\gamma \subset \Surface$. Thus
$$
e(\xi_\epsilon)=\left[ \chi(\Surface_+) - \chi(\Surface_-) \right] 
\PD[S^1 \times \{*\}].
$$
\end{remark}

The following assertion can be checked by a routine computation.
\begin{lemma}
The Reeb vector field $X_{\lambda_\epsilon}$ on 
$S^1 \times (\Surface_\pm \setminus \uU_\pm)$ is given by
\begin{equation}
\label{eqn:Reeb2}
X_{\lambda_\epsilon} = \frac{1}{1 + \epsilon^2 \left( | \nabla h_\pm |_{g_\pm}^2
 - h_\pm \right)} \left( \pm \frac{\p}{\p \phi} + 
 \epsilon j_\pm \nabla h_\pm \right). \qquad \qed
\end{equation}
\end{lemma}
In particular, this shows that every critical point $z \in \Crit(h_\epsilon)$
gives rise to a periodic orbit
$$
\gamma_z := S^1 \times \{ z \}
$$
of~$X_{\lambda_\epsilon}$.  We shall denote by $\gamma^n_z$ the
$n$-fold cover of $\gamma_z$ for any $n \in \NN$ and $z \in \Crit(h_\epsilon)$.
Observe that there is always a natural
trivialization of the contact bundle along $\gamma_z^n$, defined
by choosing any frame at a point and transporting by the $S^1$-action.

We next define a compatible complex
structure $J_\epsilon$ on $\xi_\epsilon$ as follows.  On 
$S^1 \times (\Surface_\pm \setminus \Gamma)$,
the projection $S^1 \times \Surface \to \Surface$ defines a bundle isomorphism
$$
\pi_\Surface : \xi_\epsilon|_{S^1 \times (\Surface \setminus \Gamma)} \to T \Surface|_{S^1 \times
(\Surface \setminus \Gamma)},
$$
which we can use to define $J_\epsilon : \xi_\epsilon \to \xi_\epsilon$ on
$S^1 \times (\Surface_\pm \setminus \uU_\pm)$ by
\begin{equation}
\label{eqn:projection}
J_\epsilon = \pi_\Surface^*j_\pm.
\end{equation}
Since $\p_\rho \in \xi_\epsilon$ on $S^1 \times \uU$, we can now extend
$J_\epsilon$ to this region by setting
$$
J_\epsilon \p_\rho = \alpha_\epsilon(\rho)[f_\epsilon(\rho)\p_\phi - g_\epsilon(\rho)\p_\theta],
$$
for any smooth family of functions
$\alpha_\epsilon : (-\delta,\delta) \to (0,\infty)$ 
which equals $\pm 1/g_\epsilon$ near $\rho=\pm \delta$, so in particular
for $\epsilon > 0$, $J_\epsilon$ satisfies
$$
d\rho(J_\epsilon\p_\rho) = 0 \quad\text{ and }\quad 
d\lambda_\epsilon(\p_\rho,J_\epsilon\p_\rho) > 0.
$$
Extend $J_\epsilon$ to an $\RR$-invariant almost complex structure
$$
J_\epsilon : T\left(\RR \times (S^1 \times \Surface)\right) \to 
T\left(\RR\times (S^1 \times \Surface)\right)
$$
in the standard way, i.e.~by setting $J_\epsilon \p_s = X_{\lambda_\epsilon}$ where
$s$ is the $\RR$-coordinate.
Then for each $z \in \Crit(h_\epsilon)$, there is a \emph{trivial cylinder}
$$
\RR \times S^1 \to \RR \times (S^1 \times \Surface) :
(s,t) \mapsto (s,t,z),
$$
which can be reparametrized to define an embedded $J_\epsilon$-holomorphic curve
of Fredholm index~$0$.  We shall abbreviate this curve by
$\RR\times \gamma_z$, and similarly write $\RR\times \gamma_z^n$ for the
obvious $J_\epsilon$-holomorphic $n$-fold cover of $\RR \times \gamma_z$.

\begin{prop}
\label{prop:Jcomplex}
For $\epsilon \in (0,\epsilon_0]$,
suppose $x : \RR \to \Surface$ is a solution to the gradient flow equation
$\dot{x} = \nabla h_\epsilon(x)$ approaching $z_\pm \in \Crit(h_\epsilon)$
at~$\pm\infty$.  Then there exists a proper function $a : \RR \to \RR$, unique
up to a constant, such that the embedding
$$
u_x : \RR \times S^1 \to \RR \times (S^1 \times \Surface) 
: (s,t) \mapsto (a(s),t,x(s))
$$
is a $J_\epsilon$-complex curve.  Both ends of $u$ are positive if and only if
the two critical points $z_+$ and $z_-$ lie on opposite sides of the
interface.
\end{prop}
\begin{proof}
For any $z \in \Surface$, regard $\nabla h_\epsilon(z)$ 
as a vector in $T_{(\phi,z)}(S^1 \times \Surface)$ for some fixed $\phi \in S^1$, 
and observe that $\nabla h_\epsilon(z) \in (\xi_\epsilon)_z$ due to 
Prop.~\ref{prop:contactForm}.  Thus we can define an $S^1$-invariant vector 
field
$$
v(\phi,z) = J_\epsilon \nabla h_\epsilon(z),
$$
which takes values in $\xi_\epsilon$ and vanishes only at 
$S^1 \times \Crit(h_\epsilon)$.  For $z \in \Surface_\pm \setminus \uU_\pm$,
\eqref{eqn:projection} implies that $v(\phi,z)$ is a linear
combination of $j_\pm\nabla h_\epsilon(z)$ and $\p_\phi$, and the same
is true for $z \in \uU$ due to the condition $d\rho(J_\epsilon\p_\rho) = 0$.
By \eqref{eqn:Reeb1} and \eqref{eqn:Reeb2}, the Reeb vector field 
$X_{\lambda_\epsilon}$ is also a linear combination of the same two
vector fields everywhere, and is of course linearly independent of~$v$
except when the latter vanishes, from which we conclude
$$
\p_\phi \in \RR X_{\lambda_\epsilon} \oplus \RR v
$$
everywhere on $S^1 \times \Surface$.  It follows that $J_\epsilon\p_\phi$ is everywhere
a linear combination of $\p_s$ and $\nabla h_\epsilon$, so the desired
complex curves are obtained by integrating the distribution
$$
\RR \p_\phi \oplus \RR J_\epsilon\p_\phi.
$$
In particular, this generates a foliation whose leaves include an
$\RR$-invariant family of cylinders of the form $u_x$ described above for
each nontrivial gradient flow line $x : \RR\to \Surface$, and the trivial
cylinders~$\RR\times \gamma_z$ defined above for each $z \in \Crit(h_\epsilon)$.
The signs of the cylindrical ends can now be
deduced from the orientations of the Reeb orbits, using the fact that 
the orientations of $\gamma_z$ and $\gamma_\zeta$ in the $S^1$-direction 
match if and only if $z$ and $\zeta$ lie on the same side of the 
dividing set~$\Gamma$.
\end{proof}

From the proposition it follows that each of the embeddings $u_x$
is a (not necessarily $J_\epsilon$-holomorphic) parametrization of a finite energy
$J_\epsilon$-holomorphic curve, whose Fredholm index $\ind(u_x)$ is the sum of the 
Conley-Zehnder indices at its ends if both are positive, or the difference 
if one end is negative.  We shall abuse notation by identifying the
map $u_x : \RR \times S^1 \to \RR \times (S^1 \times \Surface)$ with the
unique unparametrized $J_\epsilon$-holomorphic curve it determines, and do the same with
the obvious unbranched multiple cover
$$
u_x^n(s,t) := u_x(s,nt)
$$
for each $n \in \NN$.

%
%
%
%

\begin{prop}
\label{periodsAndIndices}
Assume $h_+$ and $h_-$ are chosen so that their gradient flows are Morse-Smale
(see Remark~\ref{remark:MorseSmale}).  Then after possibly adjusting the
gluing map $\p \Surface_+ \to \p \Surface_-$, there exist functions
\begin{equation*}
\begin{split}
(0,\epsilon_0] \to (0,\infty) &: \epsilon \mapsto T_\epsilon \\
(0,\epsilon_0] \to \NN &: \epsilon \mapsto N_\epsilon
\end{split}
\end{equation*}
with $\lim_{\epsilon \to 0} T_\epsilon = \lim_{\epsilon \to 0} N_\epsilon =
+\infty$
such that the following conditions hold for all $\epsilon > 0$:
\begin{enumerate}
\item $\nabla h_\epsilon$ is Morse-Smale.
\item Every closed
orbit of $X_{\lambda_\epsilon}$ with period less than~$T_\epsilon$ 
is either in $S^1 \times \uU$ or is $\gamma_z^n$ for some
$z \in \Crit(h_\epsilon)$ and $n \le N_\epsilon$.
\item For all $n \le N_\epsilon$, $\gamma_z^n$ is nondegenerate as an orbit of 
$X_{\lambda_\epsilon}$ and has Conley-Zehnder index
\begin{equation}
\label{eqn:CZindex}
\muCZ(\gamma_z^n) = \begin{cases}
1 & \text{ if $\ind(z) = 0$ or $2$,}\\
0 & \text{ if $\ind(z) = 1$,}
\end{cases}
\end{equation}
with respect to the $S^1$-invariant trivialization of $\xi_\epsilon$
along $\gamma_z^n$, where $\ind(z)$ denotes the Morse index of~$z$.
\end{enumerate}
\end{prop}
\begin{proof}
Up to parametrization, the flow of $\nabla h_\epsilon$ matches that of 
$\nabla h_\pm$ on $\Surface_\pm \setminus \uU_\pm$ and 
$\p_\rho$ on~$\uU$.  Thus if $\nabla h_\pm$ are both Morse-Smale,
any flow lines of $\nabla h_\epsilon$ connecting two index~$1$ 
critical points must pass through~$\Gamma$, and can thus be eliminated by
a small rotation of the gluing map $\p \Surface_+ \to \p \Surface_-$.  The existence of
the function $T_\epsilon$ with $\lim_{\epsilon \to 0} T_\epsilon = \infty$
follows from \eqref{eqn:Reeb2}, as all orbits outside of $S^1 \times\uU$ 
other than the $\gamma_z^n$ for $z \in \Crit(h_\epsilon)$ correspond to
closed orbits of $j_\pm \nabla h_\pm$ in level sets of $h_\pm$, with
periods that become infinitely large as $\epsilon \to 0$.  We can then
define
$$
N_\epsilon := \max\{ n \in \NN\ |\ \text{All $\gamma_z^n$ have periods 
$< T_\epsilon$ as orbits of $X_{\lambda_\epsilon}$} \},
$$
and observe that $N_\epsilon \to \infty$ as $\epsilon \to 0$ since
the periods of $\gamma_z$ converge to~$1$.  The formula for
$\muCZ(\gamma_z^n)$ is a standard computation from Floer theory
relating Conley-Zehnder indices to Morse indices,
see for example \cite{SalamonZehnder:Morse}.
\end{proof}

We will assume from now on that the conditions of
Prop.~\ref{periodsAndIndices} are satisfied.
Then $\nabla h_\epsilon$ is Morse-Smale for all
$\epsilon \in (0,\epsilon_0]$, and
it will follow that each of the $J_\epsilon$-holomorphic cylinders $u_x$
corresponding to gradient flow lines $x : \RR \to \Surface$ between
critical points $z_-, z_+ \in \Crit(h_\epsilon)$ has positive 
Fredholm index.  Indeed, these cylinders come in five types:
\begin{enumerate}
\item
$z_- \in \Surface_-$ with index~$0$ and $z_+ \in \Surface_+$ with index~$2$: then
$\ind(u_x) = 2$ and both ends are positive.
\item
$z_-, z_+ \in \Surface_+$ with indices~$1$ and~$2$: then $\ind(u_x) = 1$ and
one end is negative.
\item
$z_-,z_+ \in \Surface_-$ with indices~$0$ and~$1$: then $\ind(u_x) = 1$ and
one end is negative.
\item
$z_- \in \Surface_-$ with index~$0$ and $z_+ \in \Surface_+$ with index~$1$: then
$\ind(u_x) = 1$ and both ends are positive.
\item
$z_- \in \Surface_-$ with index~$1$ and $z_+ \in \Surface_+$ with index~$2$: then
$\ind(u_x) = 1$ and both ends are positive.
\end{enumerate}
This classification is exactly the same for the multiply covered
cylinders $u_x^n(s,t)$ for all $n \le N_\epsilon$.
\begin{prop}
\label{prop:regular}
For every gradient flow line $x : \RR \to \Surface$, the corresponding
$J_\epsilon$-holomorphic cylinders $u_x^n$ for $n \le N_\epsilon$ 
are all Fredholm regular.
\end{prop}
\begin{proof}
By the criterion in \cite{Wendl:automatic}*{Theorem~1}, 
an immersed, connected 
finite energy $J_\epsilon$-holomorphic curve~$u$ with genus~$g$ asymptotic to
nondegenerate Reeb orbits is Fredholm regular whenever
$$
\ind(u) > 2g - 2 + \#\Gamma_0,
$$
where the integer $\#\Gamma_0 \ge 0$ denotes the number of ends at which~$u$
approaches orbits with even Conley-Zehnder index.  In the case at hand,
we always have $g=0$ and either $\ind(u)=2$ with $\#\Gamma_0 = 0$ or
$\ind(u)=1$ with $\#\Gamma_0 = 1$, so the criterion is satisfied in all cases.
\end{proof}
It follows that the embedded cylinders $u_x$ for all gradient flow lines 
$x$ on $\Surface$, together with the trivial cylinders $\RR\times \gamma_z$
for $z \in \Crit(h_\epsilon)$, form a stable
finite energy foliation in the sense of 
\cites{HWZ:foliations,Wendl:OTfol}. 

In the following, we will make use of the intersection theory for punctured
holomorphic curves, defined by Siefring \cite{Siefring:intersection}.
This theory defines an intersection number
$$
u * v \in \ZZ
$$
for any two asymptotically cylindrical maps $u, v$ from punctured Riemann
surfaces into the symplectization of a contact $3$-manifold, with the
following properties:
\begin{itemize}
\item $u * v$ is invariant under homotopies of $u$ and~$v$ through
asymptotically cylindrical maps.
\item $u * v \ge 0$ whenever both are finite energy pseudoholomorphic curves
that are not covers of the same somewhere injective curve, and
the inequality is strict if they have nonempty intersection.
\end{itemize}

\begin{lemma}
\label{lemma:homotopy}
Suppose $u$ and~$v$ are finite energy pseudoholomorphic curves in the
symplectization $\RR\times M$ of a contact manifold $(M,\xi)$, 
such that $u$ has no negative
ends, and the positive punctures $\zeta \in \Gamma_v^+$ of~$v$ are asymptotic
to Reeb orbits denoted by~$\gamma_\zeta$.  Then
$$
u * v = \sum_{\zeta \in \Gamma_v^+} u * (\RR \times \gamma_\zeta).
$$
\end{lemma}
\begin{proof}
By $\RR$-translation we can assume the image of~$u$ is contained in
$[0,\infty) \times M$, and can then homotop~$v$ through a family of
asymptotically cylindrical maps so that its intersection with
$[0,\infty) \times M$ consists only of the trivial half-cylinders
$[0,\infty) \times \gamma_\zeta$ for $\zeta \in \Gamma_v^+$.  The lemma thus
follows from the homotopy invariance of $u * v$.
\end{proof}

It is possible in general to have $u * v > 0$ even if $u$ and $v$ are disjoint
holomorphic curves: in this case intersections can ``emerge from infinity''
under generic perturbations, and excluding this typically requires the
computation of certain winding numbers.  We will only need to worry about
this in one special case:

\begin{lemma}
\label{lemma:asympInt}
For any $z \in \Crit(h_\epsilon)$, a gradient flow line $x : \RR\to \Surface$
that begins and ends on opposite sides of the interface,
and $n \le N_\epsilon$, $(\RR\times \gamma_z^n) * u_x = 0$.
\end{lemma}
\begin{proof}
The curves $\RR\times \gamma_z^n$ and $u_x$ obviously do not intersect since $x$ does not
pass through any critical points, so we only have
to check that there are no asymptotic contributions to 
$(\RR\times \gamma_z^n) * u_x$.
This is trivially true unless $z$ is one of the end points of~$x$, so assume
the latter.  Then the definition of the intersection number in 
\cite{Siefring:intersection} implies that $(\RR\times \gamma_z^n) * u_x = 0$ 
if and only if the asymptotic end of $u_x^n$ approaching $\gamma_z^n$ 
has the largest possible asymptotic winding about the orbit.  
This bound on the winding is an integer $\alpha_-(\gamma_z^n)$, which is
the winding of a particular eigenfunction of the Hessian of the contact
action functional, and was shown in \cite{HWZ:props2} to be related to
the Conley-Zehnder index by
$$
\muCZ(\gamma_z^n) = 2\alpha_-(\gamma_z^n) + p(\gamma_z^n),
$$
where $p(\gamma_z^n) \in \{0,1\}$.
Since $\muCZ(\gamma_z^n)$ is either~$0$ or~$1$ by 
Prop.~\ref{periodsAndIndices}, we conclude $\alpha_-(\gamma_z^n) = 0$,
which is obviously the same as the winding of $u_x^n$ about $\gamma_z^n$
as it approaches asymptotically.
\end{proof}

\begin{prop}
\label{prop:allCylinders}
Suppose $u : \dot{\RS} \to \RR \times (S^1 \times \Surface)$ is a finite
energy $J_\epsilon$-holomorphic curve which is not a cover of a
trivial cylinder and has all its positive ends asymptotic to
Reeb orbits of the form $\gamma_z^n$ for $z \in \Crit(h_\epsilon)$ and
$n \le N_\epsilon$.  Then $u$ is a cover of $u_x$ for some gradient flow line
$x : \RR \to \Surface$.
\end{prop}
\begin{proof}
If $u$ is neither a cover of any $u_x$ nor of a trivial cylinder
over $\gamma_z$ for some $z \in \Crit(h_\epsilon)$,
then it must have a nontrivial intersection with one of the curves~$u_x$, 
implying $u * u_x > 0$.  By a small
perturbation using positivity of intersections, we can assume also
that~$x$ is a \emph{generic} flow line, connecting an index~$0$ critical
point $z_- \in \Surface_-$ to an index~$2$ critical point $z_+ \in \Surface_+$.
Then $u_x$ has no negative ends, so $u * u_x$ is the sum of the
intersection numbers of~$u_x$ with all the positive asymptotic orbits
of~$u$ by Lemma~\ref{lemma:homotopy}.  But these are all zero
by Lemma~\ref{lemma:asympInt}, giving a contradiction.
\end{proof}

\begin{prop}
\label{prop:covers}
Suppose $x : \RR \to \Surface$ is a gradient flow line of $h_\epsilon$ and
$u : \dot{\RS} \to \RR \times (S^1 \times \Surface)$ is a $J_\epsilon$-holomorphic 
multiple cover of~$u_x$ with
covering multiplicity at most~$N_\epsilon$.  Then $\ind(u) \ge 1$, and
the inequality is strict unless the cover is unbranched,
i.e.~$u = u_x^n$ for some $n \le N_\epsilon$.
\end{prop}
\begin{proof}
The index formula for~$u$ is
$$
\ind(u) = -\chi(\dot{\RS}) + 2 c_1(u^*\xi) + \muCZ(u),
$$
where $\muCZ(u)$ is the sum of the Conley-Zehnder indices of its positive
asymptotic orbits minus those of its negative asymptotic orbits,
and $c_1(u^*\xi)$ is the relative first Chern
number of the bundle $u^*\xi \to \dot{\RS}$ with respect to the
natural trivialization of each orbit $\gamma_z^n$.  The latter vanishes
due to the $S^1$-invariance (cf.~Remark~\ref{remark:Euler}).  
For the Conley-Zehnder indices, we
use Prop.~\ref{periodsAndIndices}, distinguishing between two cases:
\begin{itemize}
\item If $x$ passes through~$\Gamma$, then both ends of $u_x$ are positive
and thus all ends of~$u$ are positive.  Moreover, the Morse-Smale condition
guarantees that $u_x$ cannot have both its ends at hyperbolic critical points
with Conley-Zehnder index~$0$, hence $\muCZ(u) \ge 1$.
\item Otherwise $u_x$ has a positive end at an elliptic critical point
$z_+$ with $\muCZ(\gamma_{z_+}) = 1$ and a negative end at a hyperbolic
critical point $z_-$ with $\muCZ(\gamma_{z_-}) = 0$, so again
$\muCZ(u) \ge 1$.
\end{itemize}
As a result, $\ind(u) \ge -\chi(\dot{\RS}) + 1$, which is strictly
greater than~$1$ unless $\dot{\RS}$ is a cylinder, in which case there are
no branch points.
\end{proof}

\begin{prop}
\label{prop:connectors}
Suppose $z \in \Crit(h_\epsilon)$ and $u : \dot{\RS} \to \RR \times
(S^1 \times \Surface)$ is a $J_\epsilon$-holomorphic multiple cover of
$\RR\times \gamma_z$ with covering multiplicity at most~$N_\epsilon$.
Then $\ind(u) \ge 0$, and the inequality is strict unless~$u$ has
exactly one positive end.
\end{prop}
\begin{proof}
If $\ind(z)=1$, then Prop.~\ref{periodsAndIndices} implies that all asymptotic
orbits of~$u$ have Conley-Zehnder index~$0$ in the natural trivialization,
hence $\ind(u) = -\chi(\dot{\RS}) \ge 0$, with equality if and only if
$\dot{\RS}$ is a cylinder, implying it has one positive and one negative
end.  Otherwise, the asymptotic orbits of~$u$ all have Conley-Zehnder index~$1$,
so if $g \ge 0$ is the genus of~$u$ and its sets of positive and negative 
punctures are denoted by $\Gamma^+$ and $\Gamma^-$ respectively, we have
\begin{equation*}
\begin{split}
\ind(u) &= -\chi(\dot{\RS}) + \#\Gamma^+ - \#\Gamma^- =
-(2 - 2g - \#\Gamma^+ - \#\Gamma^-) + \#\Gamma^+ - \#\Gamma^- \\
&= 2g - 2 + 2\#\Gamma^+ = 2g + 2\left( \#\Gamma^+ - 1 \right) \ge 0.
\end{split}
\end{equation*}
\end{proof}

\begin{remark}
\label{remark:orientations}
The moduli spaces of $J_\epsilon$-holomorphic curves
in $\RR \times (S^1 \times \Surface)$ can be oriented coherently whenever
all asymptotic orbits are nondegenerate and ``good'', see
\cites{SFT,BourgeoisMohnke}.  In particular, the spaces of cylinders
$u_x^n$ covering gradient flow lines~$x$ can be given orientations that
match a corresponding set of coherent orientations for the spaces of
Morse gradient flow lines.
\end{remark}

\subsection{Proofs of Theorems~\ref{thm:noGiroux} and~\ref{thm:higherOrder}}
\label{subsec:proofs}

The results of the previous subsection give enough information
on $J_\epsilon$-holomorphic curves in $\RR \times (S^1 \times \Surface)$ to
prove the main theorems.  Recall that the natural compactification
of the moduli space of finite energy punctured holomorphic curves
consists of holomorphic \emph{buildings}, which in general may have
multiple levels and nodes, see \cite{SFTcompactness}.

\begin{proof}[Proof of Theorem~\ref{thm:noGiroux}]
Assume $\Surface_-$ is disconnected and let $\Surface_-^1$ and $\Surface_-^2$ denote two of
its connected components.  Then we can choose the Morse functions
$h_\pm$ so that $h_-$ has exactly one index~$0$ critical point in each
of $\Surface_-^1$ and $\Surface_-^2$, denoted by $z_-^1$ and $z_-^2$ respectively,
and $h_+$ has an index~$1$ critical point $z_+ \in \Surface_+$ such that the two 
negative gradient flow lines of~$h_\epsilon$ flowing out of~$z_+$ end
at $z_-^1$ and $z_-^2$ respectively.
In particular, there is a \emph{unique} gradient flow line~$x_1$ connecting
$z_-^1$ to~$z_+$.  By Prop.~\ref{prop:allCylinders}, the set of all
$J_\epsilon$-holomorphic buildings with no negative ends and positive ends
approaching any subset of the two simply covered orbits $\gamma_{z_+}$
and $\gamma_{z_-^1}$ consists of the following:
\begin{enumerate}
\item The cylinder $u_{x_1}$ with two positive ends at $\gamma_{z_+}$
and~$\gamma_{z_-^1}$.
\item All cylinders $u_x$ corresponding to gradient flow lines~$x$
connecting~$z_-^1$ to index~$1$ critical points in~$\Surface_-^1$.  Each of these 
cylinders has one positive and one negative end, with the positive end
approaching~$\gamma_{z_-^1}$.
\end{enumerate}
Since both of these orbits are nondegenerate and all of the holomorphic
curves in question are Fredholm regular by Prop.~\ref{prop:regular},
they all survive any sufficiently small perturbation to make
$\lambda_\epsilon$ nondegenerate and $J_\epsilon$ generic, as well as the
introduction of an abstract perturbation for the holomorphic curve
equation.  The chain complex for SFT
can therefore be defined so as to contain two special generators
$q_{\gamma_{z_-^1}}$ and $q_{\gamma_{z_+}}$ such that
$\Dsft ( q_{\gamma_{z_-^1}} q_{\gamma_{z_+}} )$ is computed by counting
the $J_\epsilon$-holomorphic curves listed above
(cf.~Remark~\ref{remark:abstract}).  We claim now that for a suitable 
choice of coherent orientations, the
algebraic count of cylinders of the second type is zero.  Indeed,
the orientations can be chosen compatibly with a choice of coherent
orientations for the space of gradient flow lines
(cf.~Remark~\ref{remark:orientations}), thus the count of these cylinders
matches the count of all gradient flow lines connecting $z_-^1$ to
index~$1$ critical points in~$\Surface_-^1$.  The latter computes a part of
the term $d \langle z_-^1 \rangle$ in the Morse cohomology of~$\Surface$,
but since $z_-^1$ is the only index~$0$ critical point in~$\Surface_-^1$,
$\langle z_-^1 \rangle$ is a closed generator of the Morse cohomology,
and the claim follows.  We conclude that only the cylinder $u_{x_1}$
with two positive ends gives a nontrivial count, and thus
$$
\Dsft \left( q_{\gamma_{z_-^1}} q_{\gamma_{z_+}} \right) = \hbar.
$$
\end{proof}

Recall from Remark~\ref{rem:truncation} that if all the Reeb
orbits below some given action $T>0$ are nondegenerate, then one can
define a truncated complex $(\aA(\lambda,T)[[\hbar]],\Dsft)$.
The proof that $(V_g,\xi_k)$ has no algebraic $(k-2)$-torsion for
$k \ge 2$ depends on establishing the following criterion.
\begin{lemma}
\label{lemma:algHyperTorsionFree}
Suppose $K$ is a nonnegative integer and $(M,\xi)$ is a closed contact
manifold admitting a contact form $\lambda$, compatible almost complex 
structure~$J$ and constant $T > 0$ with the following properties:
\begin{enumerate}
\item All Reeb orbits of~$\lambda$ with period less than~$T$ are 
nondegenerate.
\item For every pair of integers $g \ge 0$ and $r \ge 1$ with
$g + r \le K + 1$, let $\overline{\mM}_{g,r}^1(J ; T)$ denote the space of all
index~$1$ connected $J$-holomorphic buildings
in $\RR\times M$ with arithmetic genus~$g$, no negative ends, 
and $r$~positive ends approaching orbits whose periods add up to less than~$T$.
Then $\overline{\mM}_{g,r}^1(J ; T)$ consists of finitely many smooth curves
(i.e.~buildings with only one level and no nodes), which are all
Fredholm regular.
\item There is a choice of coherent orientations for which the algebraic
count of curves in $\overline{\mM}_{g,r}^1(J ; T)$ is zero
whenever $g + r \le K + 1$.
\end{enumerate}
Then if $\Dsft:\aA(\lambda,T)[[\hbar]] \to
  \aA(\lambda,T)[[\hbar]]$ is defined by counting solutions to a
sufficiently small abstract perturbation of the $J$-holomorphic curve
equation, there is no $Q \in \aA(\lambda,T)[[\hbar]]$ such that
$$
\Dsft(Q) = \hbar^K + \Order(\hbar^{K+1}).
$$
\end{lemma}
\begin{proof}
We begin by observing that since all the buildings 
in $\overline{\mM}_{g,r}^1(J ; T)$ are smooth
Fredholm regular curves, the count of the corresponding moduli space of
solutions under any suitable abstract perturbation will remain~$0$
(cf.~Remark~\ref{remark:abstract}).

Recall now that $\Dsft$ has an expansion $\Dsft= \sum D_\ell \hbar^\ell$
in powers of $\hbar$, where $D_\ell$ counts (perturbed) holomorphic
curves whose genus and number of positive punctures add up to $\ell$.
The assumption (3) now guarantees that, for every $Q \in
\aA(\lambda,T)$ each term of $D_\ell(Q)$ with $\ell \leq K$ contains
at least one  $q$-variable. 
So if $Q \in \aA(\lambda,T)[[\hbar]]$ is arbitrary, we can write its
differential uniquely as
$$
\Dsft(Q) = P + \Order(\hbar^{K+1}),
$$
with $P$ a polynomial of degree at most $K$ in $\hbar$ whose
nontrivial terms each contain at least one $q$-variable. This
establishes the claim.
\end{proof}

We now fix one of our specific examples $(V_g,\xi_k)$.  The two
sides $\Surface_+$ and $\Surface_-$ of~$\Surface$ are then both connected, so we can
choose each of the functions $h_\pm : \Surface_\pm \to \RR$ to have a unique
local minimum; in this case $h_\epsilon : \Surface \to \RR$ for $\epsilon > 0$
has a unique index~$0$ critical point in~$\Surface_-$ and a unique index~$2$ 
critical point in~$\Surface_+$.
Recall that for
any $\epsilon \in (0,\epsilon_0]$, Proposition~\ref{prop:contactForm}
gives an exact symplectic cobordism  
$$
([\epsilon,\epsilon_0] \times (S^1 \times \Surface) , d(e^s\lambda_s))
$$
relating the contact forms $e^\epsilon\lambda_\epsilon$ and
$e^{\epsilon_0}\lambda_{\epsilon_0}$. 
Then for any sufficiently $C^\infty$-small function 
$F_\epsilon: S^1 \times \Surface \to \RR$, the subdomain
$$
X_\epsilon := \{ (s,m) \in \RR \times (S^1 \times \Surface) \ |\ 
\epsilon + F_{\epsilon}(m) \le s \le \epsilon_0 \}
$$
gives an exact symplectic cobordism between
$e^{\epsilon_0}\lambda_{\epsilon_0}$ and $e^\epsilon \lambda_\epsilon'$,
where $\lambda_\epsilon'$ is the perturbed contact form
$$
\lambda_{\epsilon}' := e^{F_\epsilon} \lambda_\epsilon.
$$
By Prop.~\ref{periodsAndIndices}, $\lambda_\epsilon$ has nondegenerate
orbits up to period~$T_\epsilon$ except in $S^1 \times \uU$, thus one
can choose a generic $C^\infty$-small function $F_\epsilon$ with compact support
in $S^1 \times \uU$ so that $\lambda_\epsilon'$ has \emph{only} nondegenerate
orbits up to period~$T_\epsilon$ (the fact that generic perturbations
in an open subset suffice follows from the appendix of
\cite{AlbersBramhamWendl}).  Choose a corresponding complex
structure $J_\epsilon'$ on the perturbed contact structure
$\xi_\epsilon' := \ker\lambda_\epsilon'$ such that $J_\epsilon'$ is
$C^\infty$-close to~$J_\epsilon$.  The proof of
Theorem~\ref{thm:higherOrder} now rests on the following observation.
\begin{lemma}
\label{lem:nocurves}
The assumptions of Lemma~\ref{lemma:algHyperTorsionFree} are
satisfied with $\lambda = \lambda_\epsilon'$, $J = J_\epsilon'$,
$T = T_\epsilon$ and $K = k-2$.
\end{lemma}

\begin{proof}
It will turn out that it suffices to count holomorphic buildings for the
unperturbed structure $J_{\epsilon}$, so to start with,
suppose~$u$ is an index~$1$ $J_{\epsilon}$-holomorphic building in 
$\RR \times (S^1 \times \Surface)$
with no negative ends and at most $k-1$ positive ends, asymptotic to
orbits whose periods add up to less than~$T_{\epsilon}$.  We claim that~$u$
is then a smooth curve (with only one level and no nodes), and is a
cylinder of the form $u_x^n$ for some gradient flow line
$x : \RR \to \Surface$ of $h_{\epsilon}$ and $n \le N_{\epsilon}$.
Indeed, we start by arguing that none of the asymptotic orbits 
of~$u$ can lie in the region $S^1 \times \uU$.  By
Proposition~\ref{periodsAndIndices}, all asymptotic orbits of~$u$ outside
this region are of the form $\gamma_z^n$ for $z \in \Crit(h_\epsilon)$,
and thus have trivial projections to~$\Surface$.  Moreover,
all closed Reeb orbits in $S^1 \times \uU$ project to $\uU$ as closed
curves homologous to some positive multiple of a component
of $\Gamma$, oriented as  boundary of $\Surface_+$.  It follows that
the projection of $u$ to
$\Surface$ provides a homology from the sum of these curves to zero. Since there
are $k$ components of $\Gamma$, but only at most $k-1$ ends 
of $u$,
there is at least one component of $S^1 \times \uU$ which does not
contain any asymptotics of~$u$. Using this interface component, it is
easy to construct a closed curve on $\Surface$ which has nonzero intersection
number with the projected asymptotics of $u$ in $\uU \subset \Surface$, proving that 
the sum cannot be homologous to zero. This contradiction proves our claim that
none of the asymptotics can lie in $S^1 \times \uU$. 

Now Proposition~\ref{periodsAndIndices} implies that all the asymptotic orbits
of~$u$ are of the form $\gamma_z^n$ for $z \in \Crit(h_{\epsilon})$ and
$n \le N_{\epsilon}$.  Proposition~\ref{prop:allCylinders} then implies
that every component curve in the levels of~$u$ is one of the following:
\begin{enumerate}
\item
A cover of a trivial cylinder $\RR\times \gamma_z$ for some
$z \in \Crit(h_{\epsilon})$.
\item
A cover of the cylinder $u_x$ for some gradient flow line $x : \RR \to \Surface$
of~$h_{\epsilon}$, connecting critical points of $h_\epsilon$
  on opposite sides of $\Gamma$.
\end{enumerate}
By Proposition~\ref{prop:connectors},
all curves of the first type have nonnegative index.
Proposition~\ref{prop:covers} implies in turn that all curves of the 
second type have index at least~$1$, and there must be at least one
such curve since~$u$ has no negative ends.  Since $\ind(u)=1$, it follows
that~$u$ contains exactly one curve of the second type, which 
is an unbranched cover $u_x^n$ for some gradient flow line~$x$
and $n \le N_{\epsilon}$, and all components of~$u$ that are covers of
trivial cylinders have exactly one positive end.  Combinatorially,
this is only possible if~$u$ has precisely one nontrivial connected
component, which is of the form $u_x^n$.

By Prop.~\ref{prop:regular}, the curves $u_x^n$ are all Fredholm regular,
thus they will all survive the small perturbation of $J_{\epsilon}$
to $J_{\epsilon}'$; in fact the lack of nontrivial 
$J_{\epsilon}$-holomorphic buildings means that no additional
$J_{\epsilon}'$-holomorphic buildings can appear under this perturbation.
Thus it will suffice to show that the algebraic count of the
$J_{\epsilon}$-holomorphic cylinders $u_x^n$ for $n \le N_{\epsilon}$
is zero.  For this, choose a system of coherent
orientations for the gradient flow lines of $h_{\epsilon}$, and a
corresponding system of orientations for the moduli spaces of 
$J_{\epsilon}$-holomorphic curves (see Remark~\ref{remark:orientations}).  
The relevant count of
holomorphic curves is then the same as a certain count of gradient flow
lines: we are interested namely in all index~$1$ holomorphic cylinders
$u_x^n$ for which both ends are positive, and these correspond to the
gradient flow lines~$x$ that pass through~$\Gamma$ and connect an
index~$1$ critical point on one side to an index~$0$ or~$2$ critical point
on the other.  Consider in particular the set of all gradient flow
lines that connect the unique index~$2$ critical point $z_+ \in \Surface_+$ 
to any index~$1$ critical point in~$\Surface_-$.  The count of these flow lines 
calculates part of the differential $\p \langle z_+ \rangle$ in the
Morse homology of~$\Surface$, but since there is no other critical point of
index~$2$, $\langle z_+ \rangle$ is necessarily closed in Morse homology,
implying that the relevant algebraic count of flow lines is zero.
Applying the same argument to the unique index~$0$ critical point
in~$\Surface_-$ using Morse cohomology, we find indeed that the algebraic count
of cylinders $u_x^n$ with two positive ends for any $n \le N_{\epsilon}$
vanishes.
\end{proof}

\begin{remark}
The preceding result also establishes the conditions of
Proposition~\ref{prop:ECHlowerBound} in the appendix, thus
implying the lower bound stated in Theorem~\ref{thm:ECH}.
\end{remark}

\begin{proof}[Proof of Theorem~\ref{thm:higherOrder}]
In light of Theorem~\ref{thm:torsion}, it remains to show that
$[\hbar^{k-2}]$ does not vanish in $H_*^\SFT(V_g,\xi_k)$.

We will argue by contradiction and suppose $\hbar^{k-2}$ vanishes
in $H^\SFT_*(V_g,\xi_k)$. Choose a nondegenerate contact form
$\lambda$ such that there is a topologically trivial cobordism $X$
with positive end $(V_g,\lambda)$ and negative end
$(V_g,e^{\epsilon_0}\lambda_{\epsilon_0})$. Choose all necessary
data to define $\Dsft$ on $\aA(\lambda)[[\hbar]]$ such that it
computes $H^\SFT_*(V_g,\xi_k)$. In particular, there exists $Q \in
\aA(\lambda)[[\hbar]]$ such that
$$
\Dsft(Q) = \hbar^{k-2}.
$$
Writing $Q=Q_1 + \Order(\hbar^{k-1})$, we find a polynomial $Q_1$ of degree at
most $k-2$ in $\hbar$ with the property that 
$$
\Dsft(Q_1) = \hbar^{k-2} + \Order(\hbar^{k-1}).
$$
Note that since $Q_1$ is a polynomial in the $q$-variables, there
exists some $T>0$ such that in fact $Q_1 \in \aA(\lambda,T)[[\hbar]]$. 

Now choose $\epsilon>0$ so small that $e^\epsilon T_\epsilon>T$. 
Gluing the cobordism $X_\epsilon$ constructed above to $X$, we obtain
an exact cobordism with positive end $(V_g,\lambda)$ and negative end
$(V_g,e^\epsilon\lambda_{\epsilon}')$ which according to
Remark~\ref{rem:truncation} gives rise to a chain map,
$$
\Phi_T:(\aA(\lambda,T)[[\hbar]],\Dsft) \to
(\aA(\lambda_{\epsilon}', e^{-\epsilon}T)[[\hbar]],\Dsft),
$$
where the right hand side admits the obvious inclusion into
$(\aA(\lambda_\epsilon', T_\epsilon)[[\hbar]],\Dsft)$.
But then $\Dsft\Phi_T(Q_1) = \Phi_T \Dsft(Q_1) = \hbar^{k-2} +
\Order(\hbar^{k-1})$, which contradicts
Lemmas~\ref{lemma:algHyperTorsionFree} and~\ref{lem:nocurves}. This
contradiction shows that $\hbar^{k-2}$ cannot vanish in
$H_*^\SFT(V_g,\xi_k)$, completing the proof of the theorem.
\end{proof}

\begin{remark}
\label{remark:Jeremy2}
We conclude this section by giving the rough idea of how to
construct the exact cobordisms with positive end
$(V_{g+1},\xi_{k+1})$ and negative end $(V_{g},\xi_{k})$ 
alluded to in Remark~\ref{remark:Jeremy1}; this was explained
to us by J.~Van Horn-Morris. First observe that
if $V_g = S^1 \times \Surface$ with $\Surface = \Surface_+ \cup_\Gamma \Surface_-$ and
$V_{g+1} = S^1 \times \Surface'$ with $\Surface' = \Surface_+' \cup_{\Gamma'} \Surface_-'$, then
one can transform the former to the latter by picking two distinct points
$p_-,p_+$ in the same connected component of~$\Gamma$ and attaching
$2$-dimensional $1$-handles $\handle := \DD^1 \times \DD^1$ along the
corresponding points in both $\p \Surface_+$ and $\p \Surface_-$, 
producing $\Surface_+'$ and $\Surface_-'$ respectively with a preferred orientation
reversing diffeomorphism $\p \Surface_+' \to \p \Surface_-'$.  A Stein cobordism
between $(V_g,\xi_k)$ and $(V_{g+1},\xi_{k+1})$ is then constructed by
``multiplying the handle attachment by an annulus''.  More precisely,
we define the two Legendrian loops 
$\ell_\pm = S^1 \times \{p_\pm\} \subset V_g$,
and attach to these a $4$-dimensional \emph{round $1$-handle}
$$
\widehat{\handle} := \handle \times [-1,1] \times S^1 \cong
\DD^1 \times \left( \DD^2 \times S^1 \right)
$$
with boundary
$$
\p\widehat{\handle} = \p_-\widehat{\handle} \cup \p_+\widehat{\handle} :=
\Big(\p \DD^1 \times \left( \DD^2 \times S^1 \right) \Big) \cup
\Big( \DD^1 \times \p\left( \DD^2 \times S^1 \right) \Big).
$$
This produces a smooth cobordism from $V_g$ to $V_{g+1}$, and one can
make it into a Stein cobordism by regarding $\widehat{\handle}$ as an 
``$S^1$-invariant Weinstein handle'', with a Morse-Bott plurisubharmonic 
function with critical set $\{(0,0)\} \times S^1$, isotropic unstable
manifold $\DD^1 \times \{0\} \times S^1$ and coisotropic stable
manifold $\{0\} \times \DD^2 \times S^1$.  Perturbing the Morse-Bott
function to a Morse function with critical points of index~$1$ and~$2$
along $\{(0,0)\} \times S^1$,
one sees that the same cobordism can be obtained by attaching a
combination of standard Stein $1$-handles and $2$-handles.  One can then
use open book decompositions \cite{Vanhorn:private} to show that the 
resulting contact structure on $V_{g+1}$ is the one determined by 
the dividing curves $\Gamma' \subset \Surface'$.
\end{remark}

\section{Outlook}
\label{sec:open}

We close by mentioning a few 
questions that arise from the results of this paper.

As shown in the appendix, algebraic torsion in dimension three seems to be
closely related to the ECH contact invariant; indeed, all of our examples
are contact manifolds for which the latter vanishes, and 
they exhibit a correspondence between the minimal
order of algebraic torsion and the integers $f$ and $f_{\text{simp}}$
defined by Hutchings. It is unclear however whether
a precise relationship between these invariants exists in general, as
SFT counts a much larger class of holomorphic curves than ECH.

It is presumably also
possible to define a corresponding invariant in Heegaard Floer homology,
but the latter is apparently still unknown.

\begin{question}
Is there a Heegaard Floer theoretic contact invariant that implies 
obstructions to Stein cobordisms between pairs of contact $3$-manifolds whose
Ozsv\'ath-Szab\'o invariants vanish?
\end{question}

\begin{remark}
There is an obvious Stein cobordism obstruction in Heegaard Floer homology,
defined in terms of the largest integer~$k \ge 1$ for which the contact
invariant is in the image of the $k$th power of the so-called $U$-map.
(Note that one could define an exact cobordism obstruction in ECH
in precisely the same way.)  Nontrivial examples of this obstruction have been
computed by 
Karakurt \cite{Karakurt:Stein}.
Interestingly, since this invariant is only really interesting in cases
where the contact invariant is \emph{nonvanishing}, Karakurt's results
are completely disjoint from ours.
\end{remark}

In contrast to ECH or Heegaard Floer homology, SFT is also well defined
in higher dimensions, and it remains to find interesting examples beyond
the $0$-torsion examples that are known from 
\cites{BourgeoisNiederkrueger:PS,BourgeoisVanKoert}.  
Some candidates arise in \cite{MassotNiederkruegerWendl}:
in particular, the authors define a higher-dimensional generalization of
Giroux torsion which obstructs strong fillability and conjecturally
implies algebraic $1$-torsion.  They also find examples of contact forms
in all dimensions that have this form of torsion but don't admit any
contractible Reeb orbits, implying
there is no algebraic $0$-torsion, and in some cases the examples are also
known to be weakly (and hence stably) fillable, implying that they
do not have any fully twisted algebraic torsion.

\begin{conju}
For all integers $k \ge 1$ and $n \ge 2$, there exist infinitely many
closed $(2n-1)$-dimensional contact manifolds that have algebraic torsion
of order~$k$ but not $k-1$.  There also exist $(2n-1)$-dimensional
contact manifolds that have (untwisted) algebraic $k$-torsion but admit stable
symplectic fillings.
\end{conju}

Finally, one wonders to what extent algebraic torsion might also give
obstructions to \emph{non-exact} cobordisms.  Results in 
\cite{Wendl:cobordisms} show that Corollary~\ref{cor:noExact} for instance
is false without the exactness assumption, and the reason is that a non-exact
cobordism between $(M^+,\xi^+)$ and $(M^-,\xi^-)$ does not in general
imply a morphism
$$
H_*^\SFT(M^+,\xi^+) \to H_*^\SFT(M^-,\xi^-).
$$
On the other hand, if $(M^+,\xi^+)$ has algebraic torsion, then
$(M^-,\xi^-)$ clearly cannot be fillable, and as was explained
in \S\ref{sec:SFT}, a non-exact
cobordism does give a map from $H_*^\SFT(M^+,\xi^+)$ to a suitably
twisted version of $H_*^\SFT(M^-,\xi^-)$, where the twisting is defined by
a count of holomorphic curves without positive ends in the cobordism.
It is however unclear whether 
the vanishing of $[\hbar^k]$ in this twisted SFT also implies a result for
the untwisted theory.  A promising class of
test examples is provided by the so-called \emph{capping} and \emph{decoupling}
cobordisms constructed in \cite{Wendl:cobordisms},
for which the holomorphic curves without positive ends can be enumerated 
precisely.

\begin{question}
If $(M^+,\xi^+)$ and $(M^-,\xi^-)$ are related by a non-exact symplectic
cobordism
and $(M^+,\xi^+)$ has algebraic torsion of some finite order, must
$(M^-,\xi^-)$ also have algebraic torsion of some (possibly
higher) finite order?
Is there a precise relation between these orders for the
capping/decoupling cobordisms constructed in \cite{Wendl:cobordisms}?
\end{question}


\begin{appendix}

\newcommand{\mc}[1]{{\mathcal #1}}
\newcommand{\eqdef}{\;{:=}\;}
\newcommand{\op}{\operatorname}
\newcommand{\N}{{\mathbb N}}
\newcommand{\R}{{\mathbb R}}
\renewcommand{\epsilon}{\varepsilon}
\newcommand{\Ker}{\op{Ker}}
\newcommand{\Z}{{\mathbb Z}}

\renewcommand{\thesection}{(by Michael Hutchings)}
\section{ECH analogue of algebraic $k$-torsion}
\renewcommand{\thesection}{\Alph{section}}

The purpose of this appendix is to define an analogue of algebraic
$k$-torsion in embedded contact homology (ECH).  Specifically, given a
closed oriented $3$-manifold $Y$, a nondegenerate contact form
$\lambda$ on $Y$, and an almost complex structure $J$ on $\R\times Y$
as needed to define the ECH chain complex, we define a number
$f(Y,\lambda,J)\in\N\cup\{\infty\}$, which is similar to the order of
algebraic torsion.  It is not known whether this number is an
invariant of the contact manifold $(Y,\xi=\Ker\lambda)$.  Nonetheless
this number, together with some variants thereof, can be used to
reprove some of the results on nonexistence of exact symplectic
cobordisms between contact manifolds that are proved in the main paper
using algebraic torsion.  In addition, the results in this appendix do
not depend on any unpublished work: in particular we do not use any
symplectic field theory or Seiberg-Witten theory here.

\subsection{Basic recollections about ECH}

We begin by recalling what we will need to know about the definition
of ECH.

Let $Y$ be a closed oriented $3$-manifold with a nondegenerate contact
form $\lambda$.  Let $R$ denote the Reeb vector field determined by
$\lambda$, and let $\xi=\Ker(\lambda)$ denote the corresponding
contact structure.  Choose a generic almost complex structure $J$ on
$\R\times Y$ such that $J$ is $\R$-invariant, $J(\partial_s)=R$ where
$s$ denotes the $\R$ coordinate, and $J(\xi)=\xi$, with
$d\lambda(v,Jv)\ge 0$ for $v\in\xi$.  To save verbiage below, we refer
to the pair $(\lambda,J)$ as {\em ECH data\/} for $(Y,\xi)$.  From
these data one defines the ECH chain complex $ECC(Y,\lambda,J)$ as
follows.

An {\em orbit set\/} is a finite set of pairs
$\alpha=\{(\alpha_i,m_i)\}$ where the $\alpha_i$'s are distinct
embedded Reeb orbits and the $m_i$'s are positive integers.  The
homology class of the orbit set $\alpha$ is defined by $[\alpha]\eqdef
\sum_im_i[\alpha_i]\in H_1(Y)$.  The orbit set $\alpha$ is called {\em
  admissible\/} if $m_i=1$ whenever $\alpha_i$ is hyperbolic (i.e.\
its linearized return map has real eigenvalues).  The ECH chain
complex is freely generated over $\Z$ by admissible orbit sets.

Now let $\alpha=\{(\alpha_i,m_i)\}$ and $\beta=\{(\beta_j,n_j)\}$ be
two orbit sets with $[\alpha] = [\beta] \in H_1(Y)$.

\begin{definition}
\label{def:M}
Define $\mc{M}_J(\alpha,\beta)$ to be the moduli space of
holomorphic curves $u:(\Sigma,j)\to(\R\times Y,J)$, where the domain
$\Sigma$ is a (possibly disconnected) punctured compact Riemann
surface, and $u$ has positive ends at covers of $\alpha_i$ with total
covering multiplicity $m_i$, negative ends at covers of $\beta_j$ with
total covering multiplicity $n_j$, and no other ends.  We consider two
such holomorphic curves to be equivalent if they represent the same
$2$-dimensional current in $\R\times Y$.
\end{definition}

Let $H_2(Y,\alpha,\beta)$ denote the set of relative homology classes
of $2$-chains in $Y$ with $\partial Y =
\sum_im_i\alpha_i-\sum_jn_j\beta_j$; this is an affine space over
$H_2(Y)$.  Any holomorphic curve $u\in\mc{M}_J(\alpha,\beta)$
determines a class $[u]\in H_2(Y,\alpha,\beta)$.  If $Z\in
H_2(Y,\alpha,\beta)$, define
\[
\mc{M}_J(\alpha,\beta,Z)=\{u\in\mc{M}_J(\alpha,\beta) \mid [u]=Z\}.
\]
Also the {\em ECH index\/} is defined by
\begin{equation}
\label{eqn:I}
I(\alpha,\beta,Z) \eqdef c_\tau(Z) + Q_\tau(Z) +
\sum_i\sum_{k=1}^{m_i}\op{CZ}_\tau(\alpha_i^k) -
\sum_j\sum_{k=1}^{n_j}\op{CZ}_\tau(\beta_j^k).
\end{equation}
Here $\tau$ is a trivialization of $\xi$ over the Reeb orbits
$\alpha_i$ and $\beta_j$;  $c_\tau(Z)$ denotes the relative first
Chern class of $\xi$ over $Z$ with respect to the boundary
trivializations $\tau$;  $Q_\tau(Z)$ denotes the relative
self-intersection pairing;  and $\op{CZ}_\tau(\gamma^k)$ denotes the
Conley-Zehnder index with respect to $\tau$ of the $k^{th}$ iterate of
$\gamma$.  These notions are explained in detail in \cites{Hutchings:index,
Hutchings:revisited}.
The ECH index of a holomorphic curve $u\in\mc{M}_J(\alpha,\beta)$ is defined by $I(u)\eqdef I(\alpha,\beta,[u])$.

We will need the following facts, which are proved in 
\cite{Hutchings:revisited}*{Thm.~4.15} and 
\cite{HutchingsSullivan:T3}*{Cor.~11.5}:

\begin{proposition}
\label{prop:I}
\ \\

\vspace{-12pt}
\begin{enumerate}
\item[(a)] If $u\in\mc{M}_J(\alpha,\beta)$ does not multiply cover any
  component of its image, then $\op{ind}(u)\le I(u)$, where $\op{ind}$
  denotes the Fredholm index.
\item[(b)]
If $J$ is generic and $u\in\mc{M}_J(\alpha,\beta)$, then:
\begin{itemize}
\item
$I(u)\ge 0$, with equality if and only if $u$ is $\R$-invariant (as a
current).
\item If $I(u)=1$, then $u=u_0\sqcup u_1$ where $u_1$ is embedded and
  connected, $\op{ind}(u_1)=I(u_1)=1$, and $u_0$ is $\R$-invariant (as
  a current).
\end{itemize}
\end{enumerate}
\end{proposition}

The differential $\partial$ on the ECH chain complex is now defined as
follows:  If $\alpha$ is an admissible orbit set, then
\[
\partial \alpha \eqdef \sum_\beta \sum_{\{u\in\mc{M}_J(\alpha,\beta)/\R
 \, \mid \, I(u)=1\}} \varepsilon(u) \cdot \beta.
\]
Here the sum is over admissible orbit sets $\beta$ with
$[\alpha]=[\beta]$, and $\varepsilon(u)\in\{\pm1\}$ is a sign
explained in \cite{HutchingsTaubes:gluing2}*{\S 9}.  
The signs depend on some orientation
choices, but the chain complexes for different sign choices are
canonically isomorphic to each other.  It is shown in 
\cites{HutchingsTaubes:gluing1,HutchingsTaubes:gluing2}
that $\partial$ is well-defined and (what is much harder)
$\partial^2=0$.  The homology of the chain complex is the {\em
  embedded contact homology\/} $ECH(Y,\lambda,J)$.
Note that the empty set $\emptyset$ is a legitimate generator of the
ECH chain complex, and $\partial\emptyset=0$.  The homology class
$[\emptyset]\in ECH(Y,\lambda,J)$ is called the {\em ECH contact
  invariant\/}.

Taubes has shown that $ECH(Y,\lambda,J)$ is canonically isomorphic to
a version of Seiberg-Witten Floer cohomology 
\cite{Taubes:ECH=SWF1}, and in
particular depends only on $Y$.  In addition, under this isomorphism
the ECH contact invariant depends only on $\xi$ and agrees with an
analogous contact invariant in Seiberg-Witten Floer cohomology.
However we will not need these facts here.

There is also a filtered version of ECH which is important in
applications.  If $\alpha=\{(\alpha_i,m_i)\}$ is an orbit set, define
the {\em symplectic action\/}
\[
\mc{A}(\alpha) \eqdef \sum_im_i\int_{\alpha_i}\lambda.
\]
It follows from the conditions on $J$ that the ECH differential
decreases symplectic action, i.e.\ if
$\langle\partial\alpha,\beta\rangle\neq 0$ then
$\mc{A}(\alpha)>\mc{A}(\beta)$.  Hence for each $L\in(0,\infty]$, the
submodule $ECC^L(Y,\lambda,J)$ of $ECC(Y,\lambda,J)$ generated by admissible orbit sets of
action less than $L$ is a subcomplex.  The homology of this subcomplex
is denoted by $ECH^L(Y,\lambda,J)$, and called {\em filtered ECH\/}.
Of course, taking $L=\infty$ recovers the usual ECH.

It is shown in \cite{HutchingsTaubes:Arnold2} 
that filtered ECH does not depend on $J$ (we
will not use this fact here).  However filtered ECH does depend on the
contact form $\lambda$.  In particular, if $c$ is a positive constant,
then an almost complex structure $J$ as needed to define the ECH of
$\lambda$ determines an almost complex structure (which we also denote
by $J$) as needed to define the ECH of $c\lambda$, with the same
holomorphic curves.  There is then a
canonical isomorphism of chain complexes
\begin{equation}
\label{eqn:scaling}
ECC^L(Y,\lambda,J)=ECC^{cL}(Y,c\lambda,J),
\end{equation}
induced by the obvious bijection on generators.

\subsection{The relative filtration $J_+$}

We now recall from \cite{Hutchings:revisited}*{\S 6} 
how to define a relative filtration
on ECH which is similar to the exponent of $\hbar$ in SFT.

Let $\alpha$ and $\beta$ be admissible orbit sets with
$[\alpha]=[\beta]\in H_1(Y)$, and let $Z\in H_2(Y,\alpha,\beta)$.
Similarly to \eqref{eqn:I}, define
\begin{equation}
\label{eqn:J}
J_+(\alpha,\beta,Z) \eqdef -c_\tau(Z) + Q_\tau(Z) +
\sum_i\sum_{k=1}^{m_i-1}\op{CZ}_\tau(\alpha_i^k) -
\sum_j\sum_{k=1}^{n_j-1}\op{CZ}_\tau(\beta_j^k) + |\alpha| - |\beta|.
\end{equation}
Here $|\alpha|$ denotes the cardinality of the admissible orbit set
$\alpha$.  (There is also a more general definition of $J_+$ when the
orbit sets are not necessarily admissible, but we will not need this
here.)  If $u\in\mc{M}_J(\alpha,\beta)$, define $J_+(u)\eqdef
J_+(\alpha,\beta,[u])$.  There is now the following analogue of
Proposition~\ref{prop:I}, proved in 
\cite{Hutchings:revisited}*{Prop.~6.9 and Thm.~6.6}:

\begin{proposition}
\label{prop:J}
Let $\alpha$ and $\beta$ be admissible orbit sets with $[\alpha]=[\beta]$.
\begin{enumerate}
\item[(a)]
If $u\in\mc{M}_J(\alpha,\beta)$ is irreducible and not multiply
covered and has genus $g$, then
\begin{equation}
\label{eqn:JI}
J_+(u) \ge 2\left(g-1+|\alpha|+\sum_{i}(N_i^+ - 1) + \sum_j(N_j^- - 1)\right).
\end{equation}
Here $N_i^+$ denotes the number of positive ends of $u$ at covers of
$\alpha_i$, and $N_j^-$ denotes the number of negative ends of $u$ at
covers of $\beta_j$.  Moreover, equality holds in \eqref{eqn:JI}
when $\op{ind}(u)=I(u)$.
\item[(b)]
If $J$ is generic, and if $u\in\mc{M}_J(\alpha,\beta)$, then $J_+(u)\ge 0$.
\end{enumerate}
\end{proposition}

Note that if $u$ contributes to the ECH differential, then $J_+(u)$ is
even.  (Comparing \eqref{eqn:I} and \eqref{eqn:J} shows that the
parity of $J_+(u)-I(u)$ is the parity of the number of Reeb orbits
$\alpha_i$ or $\beta_j$ that are positive hyperbolic, which is the
parity of $\op{ind}(u)$.)  Thus we can decompose the ECH differential
$\partial$ as
\begin{equation}
\label{eqn:ddecomp}
\partial = \partial_0 + \partial_1 + \partial_2 + \cdots
\end{equation}
where $\partial_k$ denotes the contribution from holomorphic curves
$u$ with $J_+(u)=2k$.

Since $J_+$ is additive under gluing 
\cite{Hutchings:revisited}*{Prop.~6.5(a)}, it
follows that $\partial_0^2=0$,
$\partial_0\partial_1+\partial_1\partial_0=0$, etc.  Thus we obtain a
spectral sequence $E^*(Y,\lambda,J)$, where $E^1$ is the homology of
$\partial_0$, and $E^2$ is the homology of $\partial_1$ acting on
$E^1$.  Let us call this the ``$J_+$ spectral sequence''.
Unfortunately this spectral sequence is not invariant under
deformation of the contact form.  The reason is that although an exact
symplectic cobordism induces a map on ECH which (up
to a given symplectic action) is induced by a chain map that somehow counts
(possibly broken) holomorphic curves 
\cite{HutchingsTaubes:Arnold2}, 
Proposition~\ref{prop:J}(b) does not generalize to exact symplectic
cobordisms.  That is, the chain map induced by a cobordism can include
contributions from multiply covered holomorphic curves with $J_+$
negative.  However we can still use the $J_+$ spectral sequence to define
a useful analogue of the order of algebraic $k$-torsion.

\subsection{The analogue of order of algebraic torsion}

Let $Y$ be a closed oriented $3$-manifold, and let $(\lambda,J)$ be
ECH data on $Y$.

\begin{definition}
Define $f(Y,\lambda,J)$ to be the smallest nonnegative integer $k$ such
that $\emptyset$ does not survive to the $E^{k+1}$ page of the spectral sequence
$E^*(Y,\lambda,J)$.  If no such $k$ exists, define
$f(Y,\lambda,J)\eqdef\infty$.
\end{definition}

Note that if there exists $x\in ECC(Y,\lambda,J)$ with
\[
(\partial_0+\cdots+\partial_k)x=\emptyset,
\]
then $f(Y,\lambda, J)\le k$.
In particular, $f(Y,\lambda,J)<\infty$ if the ECH contact
invariant vanishes.  One can use the cobordism maps on ECH defined in
\cite{HutchingsTaubes:Arnold2} 
(using Seiberg-Witten theory) to show that $f(Y,\lambda,J)$
does not depend on $J$.  However we will not need this fact here.

There are now two difficulties in using $f$ to obstruct exact
symplectic cobordisms.  First, we would like to show that if there is
an exact symplectic cobordism from $(Y_+,\lambda_+)$ to
$(Y_-,\lambda_-)$ then
\begin{equation}
\label{eqn:wish}
f(Y_+,\lambda_+,J_+)\ge f(Y_-,\lambda_-,J_-).
\end{equation}
This would imply that $f$ depends only on the contact structure and is
monotone with respect to exact symplectic cobordisms.  Unfortunately,
we cannot prove \eqref{eqn:wish} or these consequences (and we do not
know whether these are true), due to the aforementioned lack of invariance
of the spectral sequence.  Second, $f(Y,\lambda,J)$ is difficult to
compute in practice, because often one only understands the ECH chain
complex up to a given symplectic action.

To deal with the latter difficulty, we can define a filtered version
of $f$. To prepare for this, note that the $J_+$ spectral sequence has an analogue for any subcomplex of $ECC(Y,\lambda,J)$.

\begin{definition}
  Given $L\in(0,\infty]$, define $f^L(Y,\lambda,J)$ to be the
  smallest nonnegative integer $k$ such that $\emptyset$ does not survive to the $E^{k+1}$ page of the $J_+$ spectral sequence for the subcomplex $ECC^L(Y,\lambda,J)$. If no such $k$ exists, define $f^L(Y,\lambda,J)\eqdef\infty$.
\end{definition}

The following proposition can be used in calculations to give lower bounds on
$f^L$.

\begin{proposition}\label{prop:ECHlowerBound}
  Let $(\lambda,J)$ be ECH data on $Y$, and fix
  $L\in(0,\infty]$.  Let $k$ be a positive integer.  Suppose
  that the algebraic count
\[
\sum_{\{u\in\mc{M}_J(\alpha,\emptyset,Z)/\R\}}\varepsilon(u)=0
\]
whenever:
\begin{itemize}
\item
 $\alpha$ is an admissible orbit set with $\mc{A}(\alpha)<L$, and
\item
$Z\in H_2(Y,\alpha,\emptyset)$ is such that
$I(\alpha,\emptyset,Z)=1$, and
\item
curves in $\mc{M}_J(\alpha,\emptyset,Z)$
have genus~$g$ and $N_+$ positive ends with $g + N_+ \le k$.
\end{itemize}
Then $f^L(Y,\lambda,J)\ge k$.
\end{proposition}

In the third bullet point above, note that curves in $\mc{M}_J(\alpha,\emptyset,Z)$ are embedded and connected by Proposition~\ref{prop:I}(b), and then $g$ and $N_+$ are uniquely determined by $\alpha$ and $Z$.  Here $N_+$ is determined by \cite[Thm.\ 4.15]{Hutchings:revisited}, while $g$ is determined by Proposition~\ref{prop:J}(a).

\begin{proof}
  Let $\alpha$ be an admissible orbit set with $\mc{A}(\alpha)<L$ and
  let $Z\in H_2(Y,\alpha,\emptyset)$ such that
  $I(\alpha,\emptyset,Z)=1$ and $J_+(\alpha,\emptyset,Z) < 2k$.  Then
  by Proposition~\ref{prop:I}(b), curves in
  $\mc{M}_J(\alpha,\emptyset,Z)$ are embedded and connected, so by
  Proposition~\ref{prop:J}(a), such curves have $g+N_+\le k$.  Then by
  hypothesis, the algebraic count of such curves is zero.  This means
  that $\langle\partial_i\alpha,\emptyset\rangle=0$ whenever $i<k$.
\end{proof}

We now prove a weaker version of \eqref{eqn:wish}, which will still
allow us to obstruct exact symplectic cobordisms.  This requires the
following additional definitions.

\begin{definition}
  An orbit set $\alpha=\{(\alpha_i,m_i)\}$ is {\em
    simple\/} (with respect to $J$) if:
\begin{itemize}
\item
$m_i=1$ for each $i$.
\item
If $\beta=\{(\beta_j,n_j)\}$ is another orbit set (not necessarily
admissible), and if there is a (possibly broken) $J$-holomorphic curve
from $\alpha$ to $\beta$, then $n_j=1$ for each~$j$.
\end{itemize}
Given $L\in(0,\infty]$, let
$ECC^L_{\op{simp}}(Y,\lambda,J)$ denote the subcomplex of
$ECC(Y,\lambda,J)$ generated by simple admissible orbit sets $\alpha$
with $\mc{A}(\alpha)<L$.

\end{definition}

Note that even when $L=\infty$, the homology of the subcomplex
$ECC^L_{\op{simp}}$ is not invariant under deformation of $\lambda$,
as shown by the ellipsoid example in \cite{Hutchings:ICM}.

\begin{definition}
Define $f^L_{\op{simp}}(Y,\lambda,J)$ to be the smallest nonnegative
integer $k$ such that $\emptyset$ does not survive to the $E^{k+1}$ page of the $J_+$ spectral sequence for the subcomplex
$ECC^L_{\op{simp}}(Y,\lambda,J)$. If no such $k$ exists, define $f^L_{\op{simp}}(Y,\lambda,J)\eqdef\infty$.
\end{definition}

Note that $f^L_{\op{simp}}(Y,\lambda,J)\ge
f^L(Y,\lambda,J)$, because the inclusion of chain complexes induces a morphism of spectral sequences.  The main result of this appendix is now the following theorem.

\begin{theorem}
\label{thm:fm}
Let $(\lambda_\pm,J_\pm)$ be ECH data on $Y_\pm$.
Suppose there is an exact symplectic cobordism from $(Y_+,\lambda_+)$
to $(Y_-,\lambda_-)$.  Then
\[
f^L_{\op{simp}}(Y_+,\lambda_+,J_+)\ge f^L(Y_-,\lambda_-,J_-)
\]
for each $L\in(0,\infty]$.
\end{theorem}

Here is how Theorem~\ref{thm:fm} can be used in practice to obstruct
symplectic cobordisms.  Below, write $f_{\op{simp}}\eqdef
f_{\op{simp}}^\infty$.

\begin{corollary}
\label{cor:ECHobstructions}
Suppose there exists an exact symplectic cobordism from $(Y_+,\xi_+)$
to $(Y_-,\xi_-)$.  Fix ECH data $(\lambda_+,J_+)$ for
$(Y_+,\xi_+)$ and a contact form $\lambda_-'$ with
$\Ker(\lambda_-')=\xi_-$.  Fix a positive integer $k$.  Suppose that
for each $L>0$ there exist ECH data $(\lambda_-,J_-)$ for
$(Y_-,\xi_-)$ with $f^L(Y_-,\lambda_-,J_-)\ge k$ and an exact symplectic
cobordism from $(Y_-,\lambda_-')$ to $(Y_-,\lambda_-)$.  Then
$f_{\op{simp}}(Y_+,\lambda_+,J_+)\ge k$.
\end{corollary}

\begin{proof}
  The first hypothesis implies that there exist a positive constant
  $c$ and an exact symplectic cobordism from $(Y_+,c\lambda_+)$ to
  $(Y_-,\lambda_-')$.  The second hypothesis then implies that for
  each $L>0$ there exist ECH data $(\lambda_-,J_-)$ for $(Y_-,\xi_-)$
  with $f^{L}(Y_-,\lambda_-,J_-)\ge k$ and an exact symplectic cobordism
  from $(Y_+,c\lambda_+)$ to $(Y_-,\lambda_-)$.  By the scaling
  isomorphism \eqref{eqn:scaling} and Theorem~\ref{thm:fm} we have
\[
f^{c^{-1}L}_{\op{simp}}(Y_+,\lambda_+,J_+)=f_{\op{simp}}^{L}(Y_+,c\lambda_+,J_+)\ge
k.
\]
Since $L$ was arbitrary, we conclude that
$f_{\op{simp}}(Y_+,\lambda_+,J_+)\ge k$.
\end{proof}

Here is another corollary of Theorem~\ref{thm:fm} which tells us a bit
more about the meaning of $f$.

\begin{corollary}
\label{cor:OT}
Suppose $(Y,\xi)$ is overtwisted.  Then $f(Y,\lambda,J)=0$ whenever
$(\lambda,J)$ is ECH data for $(Y,\xi)$.
\end{corollary}

\begin{proof}
  The argument in the appendix to \cite{Yau:overtwisted} shows that one can find
  ECH data $(\lambda_+,J_+)$ for $(Y,\xi)$ such that there
  is an embedded Reeb orbit $\gamma$ with the following properties:
\begin{itemize}
\item
$\gamma$ has smaller symplectic action than any other Reeb orbit.
\item
There is a unique Fredholm
  index $1$ holomorphic plane $u$ in $\R\times Y$ with positive end at
  $\gamma$.
\end{itemize}
The holomorphic plane $u$ is embedded in $\R\times Y$, so $I(u)=1$
also, and $J_+(u)=0$.  This means that
$\partial_0\{(\gamma,1)\}=\pm\emptyset$.  Since $\gamma$ has minimal
symplectic action, $\{(\gamma,1)\}$ is simple.  Thus
$f_{\op{simp}}(Y,\lambda_+,J_+)=0$.  We can also assume, by
multiplying $\lambda_+$ by a large positive constant, that there is an
exact (product) symplectic cobordism from $(Y,\lambda_+)$ to
$(Y,\lambda)$.  Theorem~\ref{thm:fm} with $L=\infty$ then implies
that $f(Y,\lambda,J)=0$.
\end{proof}

One might conjecture that the converse of Corollary~\ref{cor:OT} holds:

\begin{conjecture}
\label{conj:OT}
Given a closed contact 3-manifold $(Y,\xi)$, if $f(Y,\lambda,J)=0$ for
all ECH data $(\lambda,J)$ for $(Y,\xi)$, then $(Y,\xi)$ is
overtwisted.
\end{conjecture}

\begin{remark}
  Conjecture~\ref{conj:OT} implies the well-known conjecture that if
  $(Y_-,\xi_-)$ is a closed tight contact 3-manifold, and if
  $(Y_+,\xi_+)$ is obtained from $(Y_-,\xi_-)$ by Legendrian surgery,
  then $(Y_+,\xi_+)$ is also tight.
\end{remark}

\begin{proof}
  Suppose $(Y_+,\xi_+)$ is obtained from $(Y_-,\xi_-)$ by Legendrian
  surgery.  Recall from \cite{Weinstein:surgery} that there is an exact
  symplectic cobordism from $(Y_+,\xi_+)$ to $(Y_-,\xi_-)$.  If
  $(Y_+,\xi_+)$ is overtwisted, then as explained above one can find
  ECH data $(\lambda_+,J_+)$ for $(Y_+,\xi_+)$ such that
  $f_{\op{simp}}(Y_+,\lambda_+,J_+)=0$.  Theorem~\ref{thm:fm} then
  implies that $f(Y_-,\lambda_-,J_-)=0$ for all ECH data
  $(\lambda_-,J_-)$ for $(Y_-,\xi_-)$.  If we knew
  Conjecture~\ref{conj:OT}, then we could conclude that $(Y_-,\xi_-)$
  is overtwisted.
\end{proof}

\subsection{A cobordism chain map}
\label{sec:ccm}

We now state and prove the key lemma in the proof of
Theorem~\ref{thm:fm}.

\begin{lemma}
\label{lem:ccm}
Under the assumptions of Theorem~\ref{thm:fm}, there is a chain map
\[
\Phi: ECC^L_{\op{simp}}(Y_+,\lambda_+,J_+) \longrightarrow
ECC^L(Y_-,\lambda_-,J_-)
\]
with the following properties:
\begin{enumerate}
\item[(a)]
$\Phi(\emptyset) = \emptyset$.
\item[(b)]
There is a decomposition $\Phi=\Phi_0+\Phi_1+\cdots$ such that
\begin{equation}
\label{eqn:ssm}
\sum_{i+j=k}(\partial_i\Phi_j - \Phi_i\partial_j)=0
\end{equation}
for each nonnegative integer $k$.
\end{enumerate}
\end{lemma}

\begin{proof}
The proof has four steps.

  {\em Step 1.\/} We begin with the definition of $\Phi$.  Let
  $(X,\omega)$ be an exact symplectic cobordism from $(Y_+,\lambda_+)$
  to $(Y_-,\lambda_-)$.  Let $\lambda$ be the corresponding $1$-form
  on $X$.  There exists a neighborhood $N_+\simeq (-\epsilon,0]\times
  Y_+$ of $Y_+$ in $X$ in which $\lambda = e^s\lambda_+$ where $s$
  denotes the $(-\epsilon,0]$ coordinate.  Likewise there exists a
  neighborhood $N_-\simeq [0,\epsilon)\times Y_-$ of $Y_-$ in $X$ in
  which $\lambda = e^s\lambda_-$.  We then define the ``completion''
\[
\overline{X} = ((-\infty,0]\times Y_-) \cup_{Y_-} X \cup_{Y_+}
([0,\infty)\times Y_+),
\]
with smooth structure defined using the above neighborhoods.  Choose a
generic almost complex structure $J$ on $\overline{X}$ which agrees
with $J_+$ on $[0,\infty)\times Y_+$, which agrees with $J_-$ on
$(-\infty,0]\times Y_-$, and which is $\omega$-tame on $X$.  If
$\alpha^+$ and $\alpha^-$ are orbit sets in $Y_+$ and $Y_-$
respectively, define $\mc{M}_J(\alpha^+,\alpha^-)$ to be the moduli
space of $J$-holomorphic curves in $\overline{X}$ satisfying the
obvious analogues of the conditions in Definition~\ref{def:M}.  

The crucial point in all of what follows is this:
\begin{itemize}
\item[(*)] If the orbit set $\alpha^+$ is simple, then a holomorphic
  curve in $\mc{M}_J(\alpha^+,\alpha^-)$ cannot have any multiply
  covered component.  Also, a broken holomorphic curve arising as a
  limit of a sequence of curves in $\mc{M}_J(\alpha^+,\alpha^-)$
  cannot have any multiply covered component in the cobordism level.
\end{itemize}
Note that the proof of (*) uses exactness of the cobordism to deduce
that every component of a holomorphic curve in $\overline{X}$ has at
least one positive end.

Another key point is that the definition of the ECH index $I$, and the
index inequality in Proposition~\ref{prop:I}(a), carry over directly
to holomorphic curves in $\overline{X}$, see 
\cite{Hutchings:revisited}*{Thm.~4.15}.
In particular, if $\alpha^+$ is simple and if
$u\in\mc{M}_J(\alpha^+,\alpha^-)$ has $I(u)=0$, then the index
inequality applies to give $\op{ind}(u)\le I(u)$, and since $J$ is
generic we conclude that $I(u)=0$ and $u$ is an isolated point in the
moduli space, cut out transversely.  As a result, we can define the
map $\Phi$ as follows: If $\alpha^+$ is a simple admissible orbit set
in $Y_+$ with $\mc{A}(\alpha^+)<L$, then
\begin{equation}
\label{eqn:Phi}
\Phi(\alpha^+) \eqdef
\sum_{\alpha^-}\sum_{\{u\in\mc{M}_J(\alpha^+,\alpha^-)\mid I(u)=0\}}\epsilon(u),
\end{equation}
where the first sum is over admissible orbit sets $\alpha^-$ in
$Y_-$, and $\epsilon(u)\in\{\pm1\}$ is a sign defined as in 
\cite{HutchingsTaubes:gluing2}*{\S 9}.

{\em Step 2.\/} We now show that $\Phi$ is well-defined, i.e.\ that the sum
on the right hand side of \eqref{eqn:Phi} is finite, and we also prove
that $\Phi$ satisfies property (a).

To start, note that if there exists $u\in\mc{M}_J(\alpha^+,\alpha^-)$,
then exactness of the cobordism and Stokes's theorem imply that
$\mc{A}(\alpha^+)\ge \mc{A}(\alpha^-)$, with equality only if $u$ is
the empty holomorphic curve.  This has three important consequences.
First, $\Phi$ maps $ECC_{\op{simp}}^L$ to $ECC^L$ as required.
Second, $\Phi(\emptyset)=\emptyset$.  (The sign here follows from the
conventions in \cite{HutchingsTaubes:gluing2}*{\S 9}.)  
Third, for any simple
admissible orbit set $\alpha^+$, only finitely many admissible orbit
sets $\alpha^-$ can make a nonzero contribution to the right hand side
of \eqref{eqn:Phi}.  So to prove that $\Phi$ is well-defined, we need
to show that if $\alpha^+$ is a simple admissible orbit set in $Y_+$
and if $\alpha^-$ is an admissible orbit set in $Y_-$, then there are
only finitely many curves $u\in\mc{M}_J(\alpha^+,\alpha^-)$ with
$I(u)=0$.

Suppose to obtain a contradiction that there are infinitely many such
curves.  By a Gromov compactness argument as in 
\cite{Hutchings:index}*{Lem.~9.8}
we can then pass to a subsequence that converges to a broken
holomorphic curve with total ECH index and total Fredholm index both
equal to $0$.  By (*), the level of the broken curve in $\overline{X}$
cannot contain any multiply covered component.  Consequently the index
inequality implies that this level has $I\ge 0$, and so by
Proposition~\ref{prop:I}(a) all levels have $I=0$.  The proof of
\cite{HutchingsTaubes:gluing1}*{Lem.~7.19} 
then shows that there is only one level (i.e.\
there cannot be symplectization levels containing branched covers of
$\R$-invariant cylinders).  Thus the limiting curve is also an element
of $\mc{M}_J(\alpha^+,\alpha^-)$ with $I=0$, and since this is an
isolated point in the moduli space we have a contradiction.

{\em Step 3.\/} We now show that $\Phi$ is a chain map.  If $\alpha^+$
is a simple admissible orbit set in $Y_+$, then to prove that
$(\partial\Phi-\Phi\partial)\alpha^+=0$, we analyze ends of the $I=1$
part of $\mc{M}_J(\alpha^+,\alpha^-)$ where $\alpha^-$ is an
admissible orbit set in $Y_-$.  Again, by (*), a broken curve arising
as a limit of such curves cannot contain a multiply covered component
in the cobordism level.  Thus the proof of 
\cite{HutchingsTaubes:gluing1}*{Lem.~7.23}
carries over to show that a broken curve arising as a limit of such
curves consists of an $\op{ind}=I=0$ piece $u_0$ in the cobordism
level, an $\op{ind}=I=1$ piece $u_1$ in a symplectization level, and
(if $u_1$ is in $\R\times Y_-$) possibly additional levels in
$\R\times Y_-$ between $u_0$ and $u_1$ consisting of branched covers
of $\R$-invariant cylinders.  The gluing analysis to prove that
the ECH differential has square zero 
\cite{HutchingsTaubes:gluing1}*{Thm.~7.20} then
carries over with minor modifications to prove that
$\partial\Phi=\Phi\partial$.

{\em Step 4.\/} We now show that $\Phi$ satisfies property (b).  To do
so, note that if $u$ is a holomorphic curve counted by $\Phi$, then
$J_+(u)$ is even by the same parity argument as before.  Also, since
$u$ contains no multiply covered component, and since every component
of $u$ has a positive end, the proof of 
\cite{Hutchings:revisited}*{Thm.~6.6} carries
over to show that $J_+(u)\ge 0$.  We now define $\Phi_k$ to be the
contribution to $\Phi$ from curves $u$ with $J_+(u)=2k$.  Equation
\eqref{eqn:ssm} then follows from the fact that $J_+$ is additive
under gluing.
\end{proof}

\subsection{Conclusion}

\begin{proof}[Proof of Theorem~\ref{thm:fm}.]
 Let $^LE^*_{\op{simp}}(Y_+,\lambda_+,J_+)$ denote the $J_+$ spectral sequence for the subcomplex $ECC^L_{\op{simp}}(Y_+,\lambda_+,J_+)$, and let $^LE^*(Y_-,\lambda_-,J_-)$ denote the $J_+$ spectral sequence for the subcomplex 
  $ECC^L(Y_-,\lambda_-,J_-)$.
By Lemma~\ref{lem:ccm}(b), $\Phi$
  induces a morphism of spectral sequences
\[
\Phi^*:{^LE^*_{\op{simp}}(Y_+,\lambda_+,J_+)} \longrightarrow
{^LE^*(Y_-,\lambda_-,J_-)},
\]
which by Lemma~\ref{lem:ccm}(a) sends $\emptyset$ to $\emptyset$.  If
$f^L_{\op{simp}}(Y_+,\lambda_+,J_+)=k<\infty$, then $\emptyset$ does
not survive to ${^LE^{k+1}_{\op{simp}}}$.  Applying the morphism
$\Phi^*$ then shows that $\emptyset$ does not survive to
${^LE^{k+1}(Y_-,\lambda_-,J_-)}$, so $f^L(Y_-,\lambda_-,J_-)\le k$.
\end{proof}

\end{appendix}

\paragraph{\textbf{Acknowledgments.}}
I thank Patrick Massot for helpful discussions, and Cagatay Kutluhan and
Jeremy Van Horn-Morris for catching a minor error in the original version
of this appendix.

\end{document}